\documentclass[9pt,shortpaper,twocolumn,web]{ieeecolor}
\usepackage{generic}
\usepackage{cite}
\usepackage{amsmath,amssymb,amsfonts}
\usepackage{graphicx}
\usepackage{textcomp}

\usepackage{mathtools}
\usepackage{relsize}
\usepackage{tikz}
\graphicspath{{./Figures/}} 

\usepackage[margin=0.8cm]{caption}
\newcommand{\norm}[1]{\left\lVert#1\right\rVert}
\usepackage{pgfplots}
\pgfplotsset{compat=newest}
\usepackage{tabularx}
\pgfplotsset{plot coordinates/math parser=false}

\usepackage{pgfplots}
\pgfplotsset{compat=newest}
\usepackage{tabularx}
\pgfplotsset{plot coordinates/math parser=false}

\usepackage{algorithm}
\usepackage{soul}
\usepackage{dsfont}
\usepackage{tabto}
\usepackage[normalem]{ulem}

\usepackage{filecontents}
\usetikzlibrary{positioning}
\usetikzlibrary{calc}
\usepackage{siunitx}
\usepackage{cite}

\newtheorem{lemma}{Lemma}
\newtheorem{theorem}{Theorem}
\newtheorem{assumption}{Assumption}
\newtheorem{corollary}{Corollary}

\newtheorem{remark}{Remark}

\newtheorem{proposition}{Proposition}

\usepackage{lineno,amssymb,subcaption,algpseudocode}

\newcommand{\FS}[1][\epsilon^x]{{\mathbb{X}(#1)}}
\newcommand{\WS}[1][\epsilon^w]{{\mathbb{W}(#1)}}

\newcommand{\za}{\mathrm{z}}

\newcommand{\A}{A}
\newcommand{\B}{B}
\newcommand{\C}{C}
\newcommand{\D}{D}

\newcommand{\R}{\mathbb{R}}

\newcommand{\Y}{\mathcal{Y}}
\newcommand{\W}{\mathcal{W}}
\newcommand{\X}{\mathcal{X}}

\newcommand{\csf}{\boldsymbol{c}}
\newcommand{\bsf}{\boldsymbol{b}}
\newcommand{\dsf}{\boldsymbol{d}}
\newcommand{\fsf}{\boldsymbol{f}}
\newcommand{\qsf}{\boldsymbol{q}}
\newcommand{\lsfb}{\boldsymbol{n}}
\newcommand{\msfb}{\boldsymbol{p}}
\newcommand{\gsfb}{\boldsymbol{g}^{\mathrm{O}}}
\newcommand{\lsfy}{\boldsymbol{l}}
\newcommand{\msfy}{\boldsymbol{m}}

\newcommand{\1}{\mathbf{1}}
\newcommand{\0}{\mathbf{0}}
\newcommand{\I}{\mathbf{I}}
\usepackage{accents}

\usepackage{cuted}
\usepackage{amsmath,amssymb,amsfonts}
\usepackage{graphicx}
\usepackage{paralist}
\usepackage{rotating}
\usepackage{verbatim}

\usepackage{tikz}
\usetikzlibrary{calc,patterns,decorations.pathmorphing,decorations.markings}
\usepackage{amsmath} 
\usepackage{amssymb}  
\usepackage{mathtools}
\usepackage{graphicx}
\usepackage{dsfont}
\usepackage{bm}
\usepackage{xcolor}
\definecolor{blue_set}{RGB}{204,229.5,255}
\definecolor{pink_set}{RGB}{255,204,229.5}
\definecolor{green_set}{RGB}{229.5,255,204}
\definecolor{grey_set}{RGB}{102,102,102}
\definecolor{cyan_set}{RGB}{69,243,248}
\definecolor{yellow_set}{RGB}{229.5000,229.5000,25.5000}
\definecolor{green_border}{RGB}{0,255,0}
\definecolor{set_blue}{RGB}{25.5000,153.0000,204.0000}
\definecolor{set_green}{RGB}{25.5000,153.0000,25.5000}

\definecolor{Y_clr}{RGB}{204,204,204}
\definecolor{Y1_clr}{RGB}{40,39,38}
\definecolor{Y2_clr}{RGB}{255,237,102}
\definecolor{W_clr}{RGB}{246,114,128}
\definecolor{X1_clr}{RGB}{0,116,63}
\definecolor{X2_clr}{RGB}{248,177,149}

\usepackage{filecontents}
\usetikzlibrary{positioning}
\usetikzlibrary{calc}
\usepackage{url}
\usepackage{siunitx}

\newcommand\munderbar[1]{\underaccent{\bar}{#1}}

\newcommand\scalemath[2]{\scalebox{#1}{\mbox{\ensuremath{\displaystyle #2}}}}

\definecolor{wheat}{rgb}{0.96,0.87,0.70}
\definecolor{mario}{rgb}{0.8,0.8,1}
\definecolor{SamComm}{rgb}{0.9,0.9,0.1}
\definecolor{ao}{rgb}{0.0, 0.5, 0.0}

\usepackage{booktabs}
\usepackage{siunitx}
\usepackage{xcolor}
\usepackage{booktabs,colortbl, array}
\usepackage{pgfplotstable}
\pgfplotsset{compat=1.8}

\definecolor{rulecolor}{RGB}{0,71,171}
\definecolor{tableheadcolor}{gray}{0.92}
\newcommand{\topline}{ %
	\arrayrulecolor{rulecolor}\specialrule{0.1em}{\abovetopsep}{0pt}%
	\arrayrulecolor{tableheadcolor}\specialrule{\belowrulesep}{0pt}{0pt}%
	\arrayrulecolor{rulecolor}}
\newcommand{\midtopline}{ %
	\arrayrulecolor{tableheadcolor}\specialrule{\aboverulesep}{0pt}{0pt}%
	\arrayrulecolor{rulecolor}\specialrule{\lightrulewidth}{0pt}{0pt}%
	\arrayrulecolor{white}\specialrule{\belowrulesep}{0pt}{0pt}%
	\arrayrulecolor{rulecolor}}
\newcommand{\bottomline}{ %
	\arrayrulecolor{white}\specialrule{\aboverulesep}{0pt}{0pt}%
	\arrayrulecolor{rulecolor} %
	\specialrule{\heavyrulewidth}{0pt}{\belowbottomsep}}%

\pgfplotstableset{normal/.style ={%
		header=true,
		string type,
		font=\addfontfeature{Numbers={Monospaced}}\small,
		column type=l,
		every odd row/.style={
			before row=
		},
		every head row/.style={
			before row={\topline\rowcolor{tableheadcolor}},
			after row={\midtopline}
		},
		every last row/.style={
			after row=\bottomline
		},
		col sep=&,
		row sep=\\
	}
}

\newcommand{\SamComm}[1]{
	\begin{center}
		\fcolorbox{SamComm}{SamComm}{\parbox[t]{0.9\linewidth}{\textbf{Sampath:} #1}}
\end{center}}

\title{\LARGE \bf
	Computation of Input Disturbance Sets for Constrained Output Reachability
}

\author{Sampath Kumar Mulagaleti, Alberto Bemporad, and Mario Zanon
	\thanks{The authors are with
		the IMT School for Advanced Studies Lucca, Piazza San Francesco 19,
		55100 Lucca, Italy. 
		Email: 
		{\tt\small s.mulagaleti@imtlucca.it}}
}

\begin{document}

	\maketitle
	\thispagestyle{empty}
	\pagestyle{empty}
	
	
	\begin{abstract}
		Linear models with additive unknown-but-bounded input disturbances are extensively used to model uncertainty in robust control systems design. Typically, the disturbance set is either assumed to be known a priori or estimated from data through set-membership identification. However, the problem of computing a suitable input disturbance set in case the set of possible output values is assigned a priori has received relatively little attention. This problem arises in many contexts, such as in supervisory control, actuator design, decentralized control, and others. 
		In this paper, we propose a method to compute input disturbance sets (and the corresponding set of states) such that the resulting set of outputs matches as closely as possible a given set of outputs, while additionally satisfying strict (inner or outer) inclusion constraints.
		We formulate the problem as an optimization problem by relying on the concept of robust invariance. The effectiveness of the approach is demonstrated in numerical examples that illustrate how to solve safe reference set and input-constraint set computation problems.
	\end{abstract}

	\begin{IEEEkeywords}
		Disturbance sets, Constrained linear systems, Invariant sets.
	\end{IEEEkeywords}

	\section{Introduction}
	The theory of set invariance plays a key role in the analysis of uncertain dynamical systems, as it provides the tools for the synthesis of robust controllers that can satisfy constraints in the presence of disturbances \cite{Blanchini2015}. Of particular interest are Robust Positive Invariant (RPI) sets \cite{Blanchini1999}, the characterization and computation of which has been a very active area of research \cite{Bertsekas1971,Bertsekas1972,Kolmanovsky1998}. RPI sets are used to provide robust stability and constraint satisfaction guarantees of various robust Model Predictive Control (MPC) and Reference Governor (RG) schemes \cite{Rawlings2009,Kouvaritakis2015,Garone2017}. 
	These guarantees are usually established using the maximal robust positive invariant (MRPI) set \cite{Kolmanovsky1998}, which is the largest RPI set included in the constraint set. 
	The minimal RPI (mRPI) \cite{Blanchini2015} set, which is the smallest RPI set for a given disturbance set \cite{Kolmanovsky1998},  is used to design trajectory tubes in robust MPC \cite{Mayne2005}, and to analyze the existence of MRPI sets. It was shown that an output-admissible RPI set exists for the system if and only if the mRPI set is included in the constraint set \cite{RakovicThesis}. In order to enforce this inclusion, several methods were proposed in the literature to design feedback controllers that sufficiently attenuate the effects of disturbances \cite{Rakovic2005RCI,Riverso2013}. On the other hand, in applications such as fault-tolerant control \cite{Olaru2008}, RPI sets that include a given set are computed and used for sensor fault isolation. All the aforementioned applications were developed under the assumption that \emph{the disturbance set is known a priori}.

	 In many practical cases, however, while the set of admissible states can be estimated from sensor measurements or pre-specified from given constraints to be satisfied, \emph{the disturbance set is unknown}, leaving the designer the task of suitably defining it, especially in case one must satisfy a given set of constraints on the system, e.g., encoding known physical limitations, or undesired states. Depending on the considered setting, one might be interested in designing a disturbance set which is either as large as possible or as small as possible, while guaranteeing that the prescribed constraints are satisfied in all circumstances.
%
	 	For example, in a decentralized MPC application such as that presented in \cite{Riverso2013,Mulagaleti2021}, the dynamic coupling between subsystems is modeled as an additive disturbance. Then, since the disturbance is a combination of the states of the neighbors, it is desirable to obtain a \textit{large} disturbance set, since this implies that the coverage of the available constraint space is maximized. Another example in which a \textit{large} disturbance set is desired in presented in \cite{Flores2008}, where the disturbance set represents the set of feasible tracking references. On the other hand, if the disturbance set represents the inputs that can be applied to the system and one wants to design the actuators such that the system will be able to reach a pre-specified set of outputs, computing the \emph{smallest} disturbance set is of interest. 
	 	Moreover, in disturbance identification techniques such as those presented in \cite{Odelson2006,Mulagaleti2020a}, one is interested in obtaining a \textit{small} disturbance set that can explain the data. 
%
%
	
	In this paper, we propose a method to compute a set of input disturbances acting on a dynamical system such that the resulting output set approximately matches an assigned one. Possible applications include, but are not limited to: supervisory control and decentralized MPC design, to enforce that the output of the system stays within a given set; actuator design, to size the range of the actuators so that the output covers a given range; the characterization of the robustness of a system with respect to external disturbances. This method is centered on the formulation of an optimization problem, with the input disturbance set being the unknown and the approximation error between the obtained and assigned output sets being the objective function to minimize.
	
	We propose the formulation of the optimization problem for linear systems and polytopic sets: since the construction of the output set requires the computation of an RPI set, we extend the results of~\cite{Rakovic2013,Trodden2016} to encode the computation of a minimal parametrized polytopic RPI set within the optimization problem. Then, we propose to use the penalty-function method presented in \cite{Anandalingam1990} to solve the resulting bilevel linear program. Finally, we show the effectiveness of the approach through numerical examples related to safe reference set and input-constraint set computation problems.
	The paper is organized as follows. We introduce some notation and recall basic definitions regarding set operations in Section~\ref{sec:preliminary}. Then, we introduce the problem we solve, along with relevant results to obtain a bilevel programing formulation in Section~\ref{sec:prob_def}. In Section~\ref{sec:RPI_constraint} we present the main results that permit the implementation of the RPI constraint. In Section~\ref{sec:inclusion_con}, we discuss the encoding of the inclusion constraints, following which in Section~\ref{sec:SQP-GS}, we discuss the implementation of the penalty function method to solve the bilevel LP. Finally, in Section~\ref{sec:numerical_example} we present some numerical results along with some application demonstrations. 
	\section{Notation and Set Operations}
	\label{sec:preliminary}
	Consider the sets $\mathcal{X}, \mathcal{Y} \subset \R^{n}$, and vectors $a \in \R^{n_a}$ and $b \in \R^{n_b}$.
	Given a matrix $L\in \R^{n \times m}$, we denote by $L\mathcal{X}$ the image $\{y\in\R^{m}: y=Lx, x \in \mathcal{X}\}$ of $\mathcal{X}$ under the linear transformation induced by $L$.
	We denote the $i$-th row of matrix $L$ by $L_i$, the rank of $L$ by $\mathrm{rank}(L)$, the image-space of $L$ by $\mathrm{Im}(L)$, and the null-space of $L$ by $\mathrm{null}(L)$.
	Given a square matrix $L \in \R^{n \times n}$, $\rho(L)$ denotes its spectral radius. 
	The set $\mathcal{B}_p^n:=\{x:\norm{x}_{p}\leq 1\}$ is the unit $p$-norm ball in $\R^n$. 
	A polyhedron is the intersection of a finite number of half-spaces, and a polytope is a compact polyhedron. 
	Given two matrices $L$, $M\in \R^{n \times m}$, $L \leq M$ denotes element-wise inequality.
	The symbols $\textbf{1}$, $\textbf{0}$, and $\I$ denote all-ones, all-zeros and identity matrix respectively, with dimensions specified if the context is ambiguous. The set of natural numbers between two integers $m$ and $n$, $m\leq n$, is denoted by $\mathbb{I}_m^n:=\{m,\ldots,n\}$.
	The Minkowski set addition is defined as $\mathcal{X} \oplus \mathcal{Y}:=\{x+y:x\in\mathcal{X},y\in\mathcal{Y}\}$.
	 The Cartesian product is defined as $\mathcal{X} \times \mathcal{Y}:=\{[x^{\top} \ y^{\top}]^{\top}:x\in\mathcal{X},y\in\mathcal{Y}\}$.
	 The support function of a compact set $\mathcal{X}\subseteq \R^{n}$ for a given $y\in\R^{n}$ is defined as $h_{\mathcal{X}}(y) := \underset{x \in \mathcal{X}}{\textrm{max}} \ y^{\top}x$. Let $\mathcal{X}$ and $\mathcal{Y}$ be polytopes in $\R^n$. Then, support functions are positively homogeneous, i.e., $h_{\alpha \mathcal{X}} (y) = \alpha h_{\mathcal{X}} (y)$  for any scalar $\alpha \geq 0$.
	%
	%
	%
	Moreover, for any vector $y \in \R^n$, we have $h_{\mathcal{X}\oplus\mathcal{Y}}(y) = h_{\mathcal{X}}(y) + h_{\mathcal{Y}}(y)$. 
	%
	The inclusion $\mathcal{X} \subseteq \mathcal{Y}$ holds if and only if $h_{\mathcal{X}}(y) \leq h_{\mathcal{Y}}(y)$ for all $y \in \mathcal{B}^n_p$.
	%
	Suppose $\mathcal{Y}:=\{x:Mx \leq b\}$, then the inclusion $\mathcal{X} \subseteq \mathcal{Y}$ holds if and only if $h_{\mathcal{X}}(M_i^{\top}) \leq h_{\mathcal{Y}}(M_i^{\top})\leq b_i$ for all  $i \in \mathbb{I}_1^{n_b},$ with $h_{\mathcal{Y}}(M_i^{\top})=b_i$ if $M$,$b$ define a minimal hyperplane
	representation of $\mathcal{Y}$. We use the Hausdorff distance between polytopes $\mathcal{X}$ and $\mathcal{Y}$ defined as $d_{\mathrm{H}}(\mathcal{X},\mathcal{Y}):=\max_{y \in \mathcal{B}^{n}_p}|h_{\mathcal{X}}(y)-h_{\mathcal{Y}}(y)|$.

	\section{Problem Definition and Approximations}
	\label{sec:prob_def}
	Consider the linear time-invariant discrete-time system
	\begin{subequations}
		\label{eq:system}
		\begin{align} 
		x(t+1)&=\A x(t)+\B w(t), \label{eq:system:state}\\
		y(t)&=\C x(t)+\D w(t), \label{eq:system:output}
		\end{align}
	\end{subequations}
	with state $x\in\R^{n_x}$, output $y\in \R^{n_y}$ and  disturbance $w\in\R^{n_w}$. 
	Given a polytopic set $\Y:=\{y:Gy \leq g\}$ of outputs with $g \in \R^{m_Y}$,
	 our goal is to compute a disturbance set $\W$ such that $\Y$ is ``reachable'' by the output $y$, in a sense which we will define precisely later. We refer to $w$ as a ``disturbance'', as it is customary in the literature of uncertain systems. Depending on the application, however, it could also represent a set of command inputs, as we will show through application examples. 
	 We work with the following standing assumption.
	
	\begin{assumption}
		\label{ass:stable}
		 Matrix A is strictly stable, i.e., $\rho(\A) < 1$. $\hfill\square$
	\end{assumption} 
   In this paper, we focus on the computation of a disturbance set 
    	$\mathcal{W}$ parametrized as the polytope $\WS:=\{w:Fw \leq \epsilon^w\}$ with $\epsilon^w \in \R^{m_W}$. We further assume that the row vectors $F_i^{\top} \in \R^{n_w}$ of matrix $F$ are given a priori, and restrict our attention to computing vector $\epsilon^w$. For simplicity, we also enforce that $\0 \in \mathcal{W}$, which is equivalent to $\epsilon^w \geq \0$. In the next section, we present a method to relax this restriction, i.e., permit the computation of a disturbance set $\mathcal{W}$ that does not contain the origin.  
    
%
	Given a disturbance set $\WS$, the \textit{forward computation} problem, which is typically tackled in the literature \cite{Kolmanovsky1998,Rakovic2005}, entails computing a suitable Robust Positive Invariant (RPI) set $\mathcal{X}:=\{x:\A x +\B w \in \mathcal{X}, \forall \ w \in \WS\}$. Of particular interest is the computation of tight RPI approximations of the \textit{minimal} RPI (mRPI) set $\mathcal{X}_{\mathrm{m}}(\epsilon^w)$, which is contained in every closed RPI set. It is given by the infinite Minkowski sum
	\begin{equation}
	\label{eq:mRPI}
	\X_{\mathrm{m}}(\epsilon^w)= \bigoplus\limits_{t=0}^{\infty} \A^t\B\WS.
	\end{equation}
	If $\epsilon^w \geq \0$, i.e., $\WS$ contains the origin, then $\X_{\mathrm{m}}(\epsilon^w)$ exists, is compact, convex and unique, and contains the origin \cite{Kolmanovsky1998}. Moreover, it is the limit of all state trajectories of \eqref{eq:system:state} under persistent disturbances $w \in \WS$ \cite{Blanchini2015}. Then, the corresponding limit set of output trajectories is $\mathcal{Y}_{\mathrm{m}}(\epsilon^w):=\C \mathcal{X}_{\mathrm{m}}(\epsilon^w)\oplus \D \WS$ as per \eqref{eq:system:output}. The set $\mathcal{Y}_{\mathrm{m}}(\epsilon^w)$ exists, is compact and convex with $\0 \in \mathcal{Y}_{\mathrm{m}}(\epsilon^w)$ if $\epsilon^w \geq \0$. 
	
	In this paper, we tackle the \textit{reverse computation} problem, i.e., given an output polytope $\mathcal{Y}$, compute the vector $\epsilon^w$ such that $\mathcal{Y}_{\mathrm{m}}(\epsilon^w)=\mathcal{Y}$.
	This problem, however, might not have a solution, i.e., there might not exist any $\epsilon^w$ satisfying the output-set equality because of either of the following two reasons. ($1$) The mRPI set $\mathcal{X}_{\mathrm{m}}(\epsilon^w)$ is not finitely determined, except in a few special cases, e.g., nilpotent systems \cite{Kolmanovsky1998}. Then, depending on the set $\mathcal{X}_{\mathrm{m}}(\epsilon^w)$ and the structure of matrix $\C$, the set $\mathcal{Y}_{\mathrm{m}}(\epsilon^w)$ might also not be finitely determined. In this case, enforcing $\mathcal{Y}_{\mathrm{m}}(\epsilon^w)=\mathcal{Y}$, with $\mathcal{Y}$ defined using a finite number of hyperplanes is not possible. ($2$) Even if the set $\mathcal{X}_{\mathrm{m}}(\epsilon^w)$ and matrix $\C$ are such that $\mathcal{Y}_{\mathrm{m}}(\epsilon^w)$ is finitely determined, its shape is in general not arbitrary, but is a function of the dynamics of system \eqref{eq:system:state} and parametrization of set $\WS$. This implies that enforcing $\mathcal{Y}_{\mathrm{m}}(\epsilon^w)=\mathcal{Y}$, with $\mathcal{Y}$ being a user-specified arbitrarily shaped polytope, might not be possible. Hence, we instead tackle the problem
	\begin{equation}
	\label{eq:orig_problem_to_solve}
	\min_{\epsilon^w \geq \0} \hspace{5pt} d_{\mathrm{H}}(\mathcal{Y}_{\mathrm{m}}(\epsilon^w),\mathcal{Y}).
	\end{equation}
This formulation includes the case $\mathcal{Y}_{\mathrm{m}}(\epsilon^w)=\mathcal{Y}$, which holds if and only if $d_{\mathrm{H}}(\mathcal{Y}_{\mathrm{m}}(\epsilon^w),\mathcal{Y})=0$.
	In the rest of this paper, we present a formulation to approximately solve Problem \eqref{eq:orig_problem_to_solve}, where the approximation results from $\mathcal{X}_{\mathrm{m}}(\epsilon^w)$ not being finitely determined.
	
		\begin{remark}
We present a brief discussion regarding uniqueness of the set $\WS$ if the equality $\mathcal{Y}_{\mathrm{m}}(\epsilon^w)=\mathcal{Y}$ holds. In case $n_y < n_x+n_w$, then there exist infinitely many solutions for the equation $[C \ D][x^{\top} w^{\top}]^{\top}=y$ for each $y \in \Y$, such that the set $\WS$ is nonunique. If instead $n_y\geq n_x+n_w$ and $\mathrm{rank}([C \ D])=n_x+n_w$, then there exists a unique pair $(x,w)$ corresponding to each $y$. By construction, this implies that $\WS$ is unique. Moreover, in case $\D=\0$, $n_y \geq n_x$, and $\mathrm{rank}(\C)=n_x$, the set $\mathcal{X}_{\mathrm{m}}(\epsilon^w)=C^{\dagger} \mathcal{Y}$ is uniquely defined, where $\C^{\dagger}$ is the left-inverse of $\C$. Then, if matrix $\B$ has a left-inverse, it can be shown that $\WS$ is unique. A unified theory that includes all these cases to establish the conditions for the uniqueness of $\WS$ is a subject of future research.  $\hfill\square$
	\end{remark}
%
	\subsection{Polytopic RPI set}
	In order to approximate the mRPI set, we consider a parametrized state-set $\FS:=\{x:Ex\leq \epsilon^x\}$ with $\epsilon^x \in \R^{m_X}$ and matrix $E$ given a priori. Then, we enforce that $\FS$ is RPI for system \eqref{eq:system:state} with disturbance set $\WS$, i.e., it satisfies the inclusion $\A \FS \oplus \B \WS \subseteq \FS$. In order for such an $\epsilon^x$ to exist, however, the matrix $E$ must satisfy some requirements, that we formulate next. 
	
	Firstly, we define the support functions
	\begin{align*}
	\csf_i\left (\epsilon^x\right )&:=h_{\A\FS[\epsilon^x]}\left (E_i^{\top}\right ), & \dsf_i\left (\epsilon^w\right ):=h_{\B\WS}\left (E_i^{\top}\right ), \\
	\bsf_i\left (\epsilon^x\right )&:=h_{\FS[\epsilon^x]}\left (E_i^{\top}\right ),
	\end{align*}
	for each $i \in \mathbb{I}_1^{m_X}$,
	such that the RPI condition is equivalent to $\csf(\epsilon^x)+\dsf(\epsilon^w)\leq \bsf(\epsilon^x)$. Without loss of generality, we assume that matrix $E$ is chosen such that $\bsf(\1)=\1$.
	Then, we make the following assumption regarding the existence of an RPI set.
	\begin{assumption}
		\label{ass:RPI_E_param}
		Matrix $E$ is chosen such that there exists an $\hat{\epsilon}^x \geq \0$ satisfying the inequality $\csf(\hat{\epsilon}^x)+\1 \leq \bsf(\hat{\epsilon}^x)$. $\hfill\square$
	\end{assumption}
	Assumption \ref{ass:RPI_E_param} implies that there exists an RPI set $\FS[\hat{\epsilon}^x]$ for the system $x(t+1)=\A x(t)+ \tilde{w}(t)$ with disturbances $\tilde{w} \in \FS[\1]$. 
    In the following result, we show that there always exists an RPI set $\FS$ for system \eqref{eq:system:state} with the disturbance set $\WS$.
    \begin{proposition}
    	\label{prop:RPI_existence_proof}
    	Suppose Assumption \ref{ass:RPI_E_param} holds, then there always exists a vector $\epsilon^x \geq \0$ for every $\epsilon^w \geq \0$ such that the RPI condition $\A \FS \oplus \B \WS \subseteq \FS$ holds.$\hfill\square$
    \end{proposition}
	\begin{proof}
		If Assumption \ref{ass:RPI_E_param} holds, then by duality of linear programs and Farkas' lemma \cite{Blanchini2015}, there exist nonnegative multiplier matrices $\hat{\Lambda}_{\mathbf{c}},\hat{\Lambda}_{\mathbf{b}} \in \R^{m_X \times m_X}$ satisfying the relationships 
		\begin{align}
		\label{eq:basic_satisfied_1}
		\hat{\Lambda}_{\mathbf{c}}^{\top} \hat{\epsilon}_x + \1 \leq \hat{\Lambda}_{\mathbf{b}}^{\top} \hat{\epsilon}_x, && \hat{\Lambda}_{\mathbf{c}}^{\top} E =E \A, && \hat{\Lambda}_{\mathbf{b}}^{\top} E =E.
		\end{align}
		For a given $\epsilon^w \geq \0$, there exists an $\epsilon^x \geq \0$ satisfying the RPI condition $\csf(\epsilon^x)+\dsf(\epsilon^w)\leq \bsf(\epsilon^x)$ if and only if there exist nonnegative multiplier matrices $\Lambda_{\mathbf{c}},\Lambda_{\mathbf{b}} \in \R^{m_X \times m_X}$ satisfying 
		\begin{align}
		\label{eq:basic_satisfied_2}
		\Lambda_{\mathbf{c}}^{\top} \epsilon_x + \dsf(\epsilon^w) \leq \Lambda_{\mathbf{b}}^{\top} \epsilon_x, && \Lambda_{\mathbf{c}}^{\top} E =E \A, && \Lambda_{\mathbf{b}}^{\top} E =E.
		\end{align}
		Then, we see that setting $\epsilon^x=\norm{\dsf(\epsilon^w)}_{\infty} \hat{\epsilon}^x$ satisfies the relationships in \eqref{eq:basic_satisfied_2} with $\Lambda_{\mathbf{c}}=\hat{\Lambda}_{\mathbf{c}}$ and $\Lambda_{\mathbf{b}}=\hat{\Lambda}_{\mathbf{b}}$.
	\end{proof}
\remark{In order to verify if Assumption \ref{ass:RPI_E_param} holds, one can solve the LP (8) in \cite{Trodden2016}: The LP is bounded if and only if Assumption \ref{ass:RPI_E_param} holds. An iterative procedure to obtain a matrix $E$ that satisfies this requirement was presented in \cite{Lorenzetti2019}. $\hfill\square$} 
\remark{The choice of polytopic parametrizations with fixed hyperplanes for the disturbance set and the corresponding RPI set is motivated primarily by their computational convenience. In particular, the choice of matrix $F$ is completely independent of system \eqref{eq:system}, while matrix $E$ must satisfy Assumption \ref{ass:RPI_E_param}, which depends on system \eqref{eq:system}. Conservatism introduced by this parametrization can potentially be reduced by also optimizing over the hyperplanes through the introduction of optimization variables $E$ and $F$. We note that the results presented in the rest of this paper continue to hold in the presence of these additional variables.  Further, alternative convex parametrizations such as ellipsoidal and zonotopic sets \cite{Blanchini2015,Althoff2016} can also be considered. Embedding the computation of small RPI sets with such parametrizations within an optimization problem is a subject of future research.
$\hfill\square$}

\subsection{Approximating Problem \eqref{eq:orig_problem_to_solve}}
Having established the existence of a polytopic RPI set $\FS$ under Assumption \ref{ass:RPI_E_param}, we will now proceed with approximating Problem \eqref{eq:orig_problem_to_solve} using this set. To that end, we first note that the inclusion $\mathcal{X}_{\mathrm{m}}(\epsilon^w) \subseteq \FS$ holds by the definition of the mRPI set. Then, the set $\mathbb{Y}(\epsilon^x,\epsilon^w):=\C \FS \oplus \D\WS$ satisfies $\mathcal{Y}_{\mathrm{m}}(\epsilon^w)\subseteq \mathbb{Y}(\epsilon^x,\epsilon^w)$ by the same definition. Based on this set, we approximate Problem \eqref{eq:orig_problem_to_solve} as the bilevel programing problem
\begin{subequations}
	\label{eq:bilevel_problem_to_solve}
\begin{align}
\min_{\epsilon^w \geq \0} & \hspace{5pt} d_{\mathrm{H}}(\mathbb{Y}(\epsilon^x,\epsilon^w),\mathcal{Y}) \\
& \begin{array}{@{}r@{}l}
\displaystyle \hspace{-12pt}\text{ s.t. } \epsilon^x = \arg\min_{\munderbar{\epsilon}^x} 
&  \label{eq:lower_level_problem_RPI} \ \ d_{\mathrm{H}}(\FS[\munderbar{\epsilon}^x],\X_{\mathrm{m}}(\epsilon^w)), \\  
\displaystyle \mathrm{ s.t. }  & \ \ \csf(\munderbar{\epsilon}^x)+\dsf(\epsilon^w)\leq \bsf(\munderbar{\epsilon}^x),
\end{array}
\end{align}
\end{subequations}
in which the disturbance set is computed by the upper-level problem, for which a corresponding RPI set $\FS$ is computed by the lower-level problem in \eqref{eq:lower_level_problem_RPI}. The lower-level problem is formulated in such a way that $\FS$ is the tightest RPI approximation of the mRPI set, as seen in the objective function $d_{\mathrm{H}}(\FS[\epsilon^x],\X_{\mathrm{m}}(\epsilon^w))$. The constraint-set of this problem is nonempty according to Proposition \ref{prop:RPI_existence_proof}, and all feasible $\munderbar{\epsilon}^x\geq\0$ since $\0 \in \WS$ implies $\0 \in \mathcal{X}_{\mathrm{m}}(\epsilon^w)\subseteq \FS[\munderbar{\epsilon}^x]$.

The rationale for formulating this problem follows from the triangle inequality: For a given $\epsilon^w\geq \0$ and $\epsilon^x \geq \0$, the inequality
\begin{align*}
d_{\mathrm{H}}(\mathcal{Y}_{\mathrm{m}}(\epsilon^w),\mathcal{Y})\leq d_{\mathrm{H}}(\mathcal{Y}_{\mathrm{m}}(\epsilon^w),\mathbb{Y}(\epsilon^x,\epsilon^w))+d_{\mathrm{H}}(\mathbb{Y}(\epsilon^x,\epsilon^w),\Y)
\end{align*}
 holds. The second part of the inequality is minimized by the upper-level problem in \eqref{eq:bilevel_problem_to_solve}. With respect to the first part, let us define $d^\mathrm{x}_{\mathrm{H}}:=d_{\mathrm{H}}(\FS,\mathcal{X}_{\mathrm{m}}(\epsilon^w))$. Then, by definition of Hausdorff distance, we have $\FS \subseteq \mathcal{X}_{\mathrm{m}}(\epsilon^w) \oplus d^\mathrm{x}_{\mathrm{H}} \mathcal{B}^{n_x}_{2}$. From basic properties of Minkowski algebra, it follows that 
 \begin{align*}
 \C\FS\oplus\D\WS \subseteq \C\mathcal{X}_{\mathrm{m}}(\epsilon^w) \oplus\D\WS \oplus d^\mathrm{x}_{\mathrm{H}} \C\mathcal{B}^{n_x}_{\infty},
 \end{align*}
  which by definition of Hausdorff distance implies $d_{\mathrm{H}}(\mathcal{Y}_{\mathrm{m}}(\epsilon^w),\mathbb{Y}(\epsilon^x,\epsilon^w)) \leq d^\mathrm{x}_{\mathrm{H}} \norm{h_{\C\mathcal{B}^{n_x}_{\infty}}(p)}_{\infty}$ for all $p \in \mathcal{B}^{n_y}_{2}$. Since $d^\mathrm{x}_{\mathrm{H}}$ is minimized by the lower-level problem \eqref{eq:lower_level_problem_RPI},
Problem \eqref{eq:bilevel_problem_to_solve} minimizes an upper bound to Problem \eqref{eq:orig_problem_to_solve} as
\begin{align}
\label{eq:Hausdorf_ineq}
d_{\mathrm{H}}(\mathcal{Y}_{\mathrm{m}}(\epsilon^w),\mathcal{Y})\leq d^\mathrm{x}_{\mathrm{H}} \norm{h_{\C\mathcal{B}^{n_x}_{\infty}}(p)}_{\infty}+d_{\mathrm{H}}(\mathbb{Y}(\epsilon^x,\epsilon^w),\Y).
\end{align}

Finally, in order to eliminate the mRPI set from Problem \eqref{eq:lower_level_problem_RPI} we use the following results from \cite{Rakovic2013}, which state that the solution of Problem \eqref{eq:lower_level_problem_RPI} can be obtained using fixed-point iterations. In recalling these results, we denote $\dsf(\epsilon^w)$ by $\dsf$ for ease of notation.
\begin{lemma} \cite[Theorems 1 and 2, Corollary 1]{Rakovic2013}
	\label{lem:brouwer}
	Suppose Assumption \ref{ass:RPI_E_param} holds and $\epsilon^w \geq \0$. Define the sets
	\begin{align*}
	\mathcal{H}(\dsf)&:=\{\epsilon^x: \0 \leq \epsilon^x \leq \norm{\dsf}_{\infty}\hat{\epsilon}^x\},\\ \mathcal{E}(\dsf)&:=\{\epsilon^x\geq \0:\csf(\epsilon^x)+\dsf\leq \bsf(\epsilon^x)\},
	\end{align*} 
	where $\mathcal{E}(\dsf)$ is the constraint-set of the lower-level problem \eqref{eq:lower_level_problem_RPI}.
	\begin{enumerate}
		\item The sequence generated by the iterative procedure $\epsilon^x_{[k+1]}:=\csf(\epsilon^x_{[k]})+\dsf$ from any initial-condition $\epsilon^x_{[0]} \in \mathcal{H}(\dsf)$
		converges to a fixed-point $\epsilon_*^x(\epsilon^x_{[0]},\dsf) \in \mathcal{H}(\dsf)$, i.e., $\epsilon_*^x(\epsilon^x_{[0]},\dsf):=\lim\limits_{k \to \infty} \epsilon^x_{[k]}$. This fixed-point satisfies the equalities
		\begin{align*}
		\scalemath{0.95}{\csf(\epsilon_*^x(\epsilon^x_{[0]},\dsf))+\dsf=\bsf(\epsilon_*^x(\epsilon^x_{[0]},\dsf)), \ \bsf(\epsilon_*^x(\epsilon^x_{[0]},\dsf))=\epsilon_*^x(\epsilon^x_{[0]},\dsf).}
		\end{align*}
		\item 
		The fixed-point reached from the initial-condition $\epsilon^x_{[0]}=\0$ satisfies the inequality $\norm{\epsilon_*^x(\0,\dsf)}_1 \leq \norm{\munderbar{\epsilon}^x}_1$ for all $\munderbar{\epsilon}^x \in \mathcal{E}(\dsf)$, and $\FS[\epsilon_*^x(\0,\dsf)]$ is the minimal parametrized RPI set, i.e., 
		\begin{align*}
		\X_{\mathrm{m}}(\epsilon^w) \subseteq \FS[\epsilon_*^x(\0,\dsf)] = \bigcap_{\munderbar{\epsilon}^x \in \mathcal{E}(\dsf)} \FS[\munderbar{\epsilon}^x].
		\end{align*} $\hfill\square$
	\end{enumerate} 
\end{lemma}

From Lemma \ref{lem:brouwer}.2, we see that $\epsilon_*^x(\0,\dsf(\epsilon^w))$ is the solution of the lower-level problem \eqref{eq:lower_level_problem_RPI}, since the RPI set $\FS[\epsilon_*^x(\0,\dsf(\epsilon^w))]$ satisfies
\begin{align*}
d_{\mathrm{H}}(\FS[\epsilon_*^x(\0,\dsf(\epsilon^w))],\mathcal{X}_{\mathrm{m}}(\epsilon^w)) \leq d_{\mathrm{H}}(\FS[\munderbar{\epsilon}^x],\mathcal{X}_{\mathrm{m}}(\epsilon^w))
\end{align*}
for all $\munderbar{\epsilon}^x \in \mathcal{E}(\dsf(\epsilon^w))$. 
 Since this solution also has the smallest $1$-norm value over all $\munderbar{\epsilon}^x \in \mathcal{E}(\dsf(\epsilon^w))$, Problem \eqref{eq:bilevel_problem_to_solve} is equivalent to
%
	\begin{align}
		\label{eq:bilevel_problem_to_solve_1_norm}
	\min_{\epsilon^w \geq \0} &\hspace{5pt} d_{\mathrm{H}}(\mathbb{Y}(\epsilon^x,\epsilon^w),\mathcal{Y}) \\
	& \begin{array}{@{}r@{}l}
	\displaystyle \hspace{-12pt}\text{ s.t. } \epsilon^x = \arg\min_{\munderbar{\epsilon}^x \in \mathcal{E}(\dsf(\epsilon^w))} 
	& \ \ \norm{\munderbar{\epsilon}^x}_1. \nonumber
	\end{array}
	\end{align}

While Problem \eqref{eq:bilevel_problem_to_solve_1_norm} minimizes the distance $d_{\mathrm{H}}(\mathbb{Y}(\epsilon^x,\epsilon^w),\mathcal{Y})$, there exist problem settings which have stronger requirements with respect to inclusions of the output-set $\mathcal{Y}$.

\textit{Example 1:
	Consider the case in which \eqref{eq:system:state} represents a linear system equipped with a stabilizing feedback controller, such that $w$ represents the feedforward reference signal. The system is subject to constraints $y \in \Y$. 
	Then, an \textit{inner-approximation} version of Problem \eqref{eq:bilevel_problem_to_solve_1_norm}  enforces $\mathcal{Y}_{\mathrm{m}}(\epsilon^w)\subseteq \Y$ to compute the set of references $\WS$ that satisfy the system constraints.
}

\textit{Example 2: Consider the case in which system~\eqref{eq:system} is used to model a dynamic disturbance
$y(t)$, 
e.g., pedestrian behavior \cite{Batkovic2018} or an asset price, possibly estimated from linear time-series analysis. Let $\Y$ represent experimental data from the disturbance-generating system. Then, an  \textit{outer-approximation} version of Problem \eqref{eq:bilevel_problem_to_solve_1_norm} enforces $\mathcal{Y}_{\mathrm{m}}(\epsilon^w)$ to compute a
disturbance set $\WS$, using which a \textit{realistic} simulator can be designed. }

In order to formulate the \textit{inner-approximation} version of Problem \eqref{eq:bilevel_problem_to_solve_1_norm}, we note that enforcing $\mathbb{Y}(\epsilon^x,\epsilon^w) \subseteq \Y$ guarantees $\mathcal{Y}_{\mathrm{m}}(\epsilon^w) \subseteq \Y$ since $\mathcal{Y}_{\mathrm{m}}(\epsilon^w) \subseteq \mathbb{Y}(\epsilon^x,\epsilon^w)$. Hence, we formulate
\begin{align}
	\label{eq:bilevel_problem_to_solve_1_norm_innerApp}
	\min_{\epsilon^w \geq \0} &\hspace{5pt} d_{\mathrm{H}}(\mathbb{Y}(\epsilon^x,\epsilon^w),\mathcal{Y}) \\
	& \hspace{-8pt}\text{ s.t. } \mathbb{Y}(\epsilon^x,\epsilon^w)\subseteq\mathcal{Y}, \nonumber \\
	& \begin{array}{@{}r@{}l}
	\displaystyle \hspace{12pt}\epsilon^x = \arg\min_{\munderbar{\epsilon}^x \in \mathcal{E}(\dsf(\epsilon^w))} 
	& \ \ \norm{\munderbar{\epsilon}^x}_1. \nonumber
	\end{array}
\end{align}

In order to formulate the \textit{outer-approximation} version of Problem \eqref{eq:bilevel_problem_to_solve_1_norm}, we note that $\oplus_{t=0}^N \A^t \B \WS \subseteq \mathcal{X}_{\mathrm{m}}(\epsilon^w)$ for all $N\geq \0$. Then, the inclusion $ \Y\subseteq \mathcal{Y}_{\mathrm{m}}(\epsilon^w)$ can be enforced by choosing an index $N\geq \0$, and appending the constraint $\Y \subseteq \oplus_{t=0}^N \C\A^t \B \WS \oplus \D \WS$. Hence, we formulate
\begin{align}
\label{eq:bilevel_problem_to_solve_1_norm_outerApp}
\min_{\epsilon^w \geq \0} &\hspace{5pt} d_{\mathrm{H}}(\mathbb{Y}(\epsilon^x,\epsilon^w),\mathcal{Y}) \\
& \hspace{-8pt}\text{ s.t. } \Y \subseteq \oplus_{t=0}^N \C\A^t \B \WS \oplus \D \WS, \nonumber \\
& \begin{array}{@{}r@{}l}
\displaystyle \hspace{12pt}\epsilon^x = \arg\min_{\munderbar{\epsilon}^x \in \mathcal{E}(\dsf(\epsilon^w))} 
& \ \ \norm{\munderbar{\epsilon}^x}_1. \nonumber
\end{array}
\end{align}

We now formulate the assumptions that the output-set $\Y$ must satisfy in order to guarantee feasibility of Problems \eqref{eq:bilevel_problem_to_solve_1_norm_innerApp},\eqref{eq:bilevel_problem_to_solve_1_norm_outerApp}
\begin{assumption}
	\label{ass:feasibility_assumptions}
	(\textit{Inner}): The origin belongs to the output-set, i.e., $\{\0\} \in \Y$; (\textit{Outer}): The output-set belongs to the output controllable subspace, i.e., 
	$\Y \subset \mathrm{Im}([\C\B \ \C\A\B \ \cdots \C\A^{n_x-1}\B \ \D])$. $\hfill\square$
\end{assumption}

Under Assumption \ref{ass:feasibility_assumptions}-\textit{Inner}, vector $g\geq\0$, and $(\epsilon^x,\epsilon^w)=\0$ are feasible solutions of Problem \eqref{eq:bilevel_problem_to_solve_1_norm_innerApp}: this condition is necessary and sufficient for the existence of a set $\WS$ solving the \textit{inner-approximation} problem.

Under Assumption \ref{ass:feasibility_assumptions}-\textit{Outer}, all $y \in \Y$ can be reached from the origin with feasible inputs $w$. Then, Problem \eqref{eq:bilevel_problem_to_solve_1_norm_outerApp} is feasible for all $N \geq n_x$: this condition is necessary and sufficient for the existence of a set $\WS$ solving the  \textit{outer-approximation} problem.
Moreover,
 if $\mathrm{rank}([\C\B \ \C\A\B \ \cdots \C\A^{n_x-1}\B \ \D])=n_y$, then Problem \eqref{eq:bilevel_problem_to_solve_1_norm_outerApp} is feasible for every nonempty $\Y$ and $N \geq n_x$ since system \eqref{eq:system} is then output-controllable.

In the rest of this paper, we transform Problems \eqref{eq:bilevel_problem_to_solve_1_norm_innerApp}-\eqref{eq:bilevel_problem_to_solve_1_norm_outerApp} into implementable forms. To that end, in the next section, we discuss the characterization of the RPI constraints.
\section{Characterization of RPI Constraints }
\label{sec:RPI_constraint}
In this section, we show that the lower-level problems in \eqref{eq:bilevel_problem_to_solve_1_norm_innerApp},\eqref{eq:bilevel_problem_to_solve_1_norm_outerApp} that characterize the minimal parametrized RPI set can be replaced by the equality $\csf(\epsilon^x)+\dsf(\epsilon^w)=\epsilon^x$, i.e., the equivalence
\begin{align}
\label{eq:main_result_RPI}
\epsilon^x = \arg\min_{\munderbar{\epsilon}^x \in \mathcal{E}(\dsf(\epsilon^w))} 
 \ \ \norm{\munderbar{\epsilon}^x}_1 && \Leftrightarrow && \csf(\epsilon^x)+\dsf(\epsilon^w)=\epsilon^x
\end{align}
holds.
For ease of notation, we denote $\dsf(\epsilon^w)$ by $\dsf$ in the rest of this section, since the results are presented for a fixed $\epsilon^w \geq \0$.

Firstly, we recall from Lemma \ref{lem:brouwer} that the fixed-point solution
\begin{align*}
\epsilon_*^x(\0,\dsf) = \arg\min_{\munderbar{\epsilon}^x \in \mathcal{E}(\dsf)} 
\ \ \norm{\munderbar{\epsilon}^x}_1
\end{align*}
exists, and satisfies the equality $\csf(\epsilon_*^x(\0,\dsf))+\dsf=\epsilon_*^x(\0,\dsf)$.
Moreover, every $\epsilon^x$ that satisfies $\csf(\epsilon^x)+\dsf=\epsilon^x$ is a fixed-point for the iteration $\epsilon^x_{[k+1]}=\csf(\epsilon^x_{[k]})+\dsf$ with $\epsilon^x_{[0]}=\epsilon^x$, i.e., $\epsilon^x=\epsilon^x_*(\epsilon^x,\dsf)$. Then, if there exists a unique fixed-point
\begin{align*}
\epsilon^x_{\#}(\dsf):=\epsilon_*^x(\epsilon_{[0]}^x,\dsf), && \forall \ \epsilon_{[0]}^x \in \mathcal{H}(\dsf),
\end{align*}
 the equivalence in \eqref{eq:main_result_RPI} holds with $\epsilon^x=\epsilon^x_{\#}(\dsf)$.
%
%
%
%
In the following result from \cite{Trodden2016}, the uniqueness of the fixed-point was shown under a slightly restrictive assumption.
\begin{lemma}\cite[Theorem~3]{Trodden2016}
	\label{lem:trodden}
	Suppose Assumption \ref{ass:RPI_E_param} holds and $\dsf > \0$, then the fixed-point $\epsilon_*^x(\epsilon^x_{[0]},\dsf)$ is unique. 
	That is, there exists an $\epsilon_{\#}^x(\dsf):=\epsilon_*^x(\epsilon^x_{[0]},\dsf)$ for all $\epsilon^x_{[0]} \in  \mathcal{H}(\dsf)$ satisfying the equality
	$\csf(\epsilon_{\#}^x(\dsf))+\dsf=\epsilon_{\#}^x(\dsf)$.
	  $\hfill\square$
\end{lemma}

We now present a brief discussion regarding the restrictions imposed by the assumption $\dsf>\0$: recalling the definition
\begin{equation*}
\dsf_i=\max_{w} E_i B w \text{ s.t. } Fw \leq \epsilon^w,
\end{equation*}
we see that $\dsf_i>0$ for all $i \in \mathbb{I}_1^{m_X}$ only if $E_i \B \neq \0$ for each $i\in \mathbb{I}_1^{m_X}$, and $\epsilon^w > \0$.  While the positivity condition $\epsilon^w>\0$ can be enforced easily through a linear constraint in Problems \eqref{eq:bilevel_problem_to_solve_1_norm_innerApp}-\eqref{eq:bilevel_problem_to_solve_1_norm_outerApp}, the former condition holds only if the additional assumption $E_i^\top \notin \mathrm{null}(B^{\top})$ (or the stronger assumption $\mathrm{rank}(B)=n_x$) is satisfied: these assumptions restrict the class of systems and RPI set parametrizations that are often encountered. Moreover, they lead to excessively conservative RPI set parametrizations. For example, an uncontrollable system would require an RPI set parametrization that always includes the origin within its interior.

%
%

In the following result, we use continuity properties of support functions to show that the uniqueness of fixed-point holds without these additional assumptions. To that end, we introduce the perturbed disturbance vector $\dsf_{\delta}:=\dsf+\delta \1$ defined for a scalar $\delta>0$.
\begin{theorem}
	\label{thm:induction_theorem}
	Suppose Assumption \ref{ass:RPI_E_param} holds and $\dsf \geq \0$, 
	then the fixed-point $\epsilon_*^x(\epsilon^x_{[0]},\dsf)$ is unique. That is, there exists an $\epsilon_{\#}^x(\dsf):=\epsilon_*^x(\epsilon^x_{[0]},\dsf)$ for all $\epsilon^x_{[0]} \in \mathcal{H}(\dsf)$ satisfying
	$\csf(\epsilon_{\#}^x(\dsf))+\dsf=\epsilon_{\#}^x(\dsf)$.
	$\hfill\square$
\end{theorem}
\begin{proof}
	Consider the fixed-point iterations 
	\begin{equation*}\small
	\epsilon^x_{[k+1]}(\dsf):=\csf(\epsilon^x_{[k]}(\dsf))+\dsf,  \ \ \epsilon^x_{[k+1]}(\dsf_{\delta}):=\csf(\epsilon^x_{[k]}(\dsf_{\delta}))+\dsf_{\delta},
	\end{equation*}
	starting from the same initial point $\scalemath{0.95}{\epsilon^x_{[0]}(\dsf), \ \epsilon^x_{[0]}(\dsf_{\delta}) := \epsilon^x_{[0]} \in \mathcal{H}(\dsf_{\delta}).}$
	We first show by induction that at all iterations $k$, the limit
	\begin{equation}
	\label{eq:k_distance}
	\lim\limits_{\delta \to 0^{+}} \norm{\epsilon^x_{[k]}(\dsf)-\epsilon^x_{[k]}(\dsf_{\delta})} = 0.
	\end{equation}
	At the first fixed-point iteration we have $\epsilon^x_{[1]}(\dsf):=\csf(\epsilon^x_{[0]})+\dsf$ and $\epsilon^x_{[1]}(\dsf_\delta):=\csf(\epsilon^x_{[0]})+\dsf_\delta$.
	Hence, by definition of $\dsf_{\delta}$,~\eqref{eq:k_distance} holds for $k=1$, i.e., $\lim\limits_{\delta \to 0^{+}} \norm{\epsilon^x_{[1]}(\dsf)-\epsilon^x_{[1]}(\dsf_{\delta})} = 0.$
	We will now proceed by induction by exploiting the fact that if \eqref{eq:k_distance} holds at iteration index $k$, then 
	\begin{equation}
	\label{eq:k_hypth}
	\epsilon^x_{[k]}(\dsf) \to \epsilon^x_{[k]}(\dsf_{\delta}) \ \ \ \text{as} \ \ \ \delta \to 0^+.
	\end{equation}
	At iteration index $k+1$ we have
	{
		\begin{align}
		\label{eq:k_plus_1_distance}
		&\lim\limits_{\delta \to 0^{+}} \norm{\epsilon^x_{[k+1]}(\dsf)-\epsilon^x_{[k+1]}(\dsf_{\delta})}  \\
		=& \lim\limits_{\delta \to 0^{+}} \norm{\csf(\epsilon^x_{[k]}(\dsf))+\dsf-\csf(\epsilon^x_{[k]}(\dsf_{\delta}))-\dsf_{\delta}} \nonumber \\
		\leq &
		\lim\limits_{\delta \to 0^{+}} \norm{\csf(\epsilon^x_{[k]}(\dsf))-\csf(\epsilon^x_{[k]}(\dsf_{\delta}))}+\lim\limits_{\delta \to 0^{+}} \delta. \nonumber
		\end{align}
	}
	Since $\csf$ is a continuous function of $\dsf$, for all $\dsf\geq0$, we have
	\begin{equation}
	\label{eq:cont_c}
	\lim\limits_{\epsilon^x_{[k]}(\dsf) \to \epsilon^x_{[k]}(\dsf_{\delta})} \norm{\csf(\epsilon^x_{[k]}(\dsf))-\csf(\epsilon^x_{[k]}(\dsf_{\delta}))} = 0.
	\end{equation}
	Hence, \eqref{eq:k_hypth} implies
	\begin{equation*}
	\lim\limits_{\delta \to 0^{+}} \norm{\epsilon^x_{[k+1]}(\dsf)-\epsilon^x_{[k+1]}(\dsf_{\delta})}=0.
	\end{equation*}
	Thus, by induction, \eqref{eq:k_distance} holds for all iteration indices $k$.
	Since Assumptions \ref{ass:RPI_E_param} holds, then according to Lemma \ref{lem:brouwer}, the iterates $\{\epsilon^x_{[k]}(\dsf)\}$ and $\{\epsilon^x_{[k]}(\dsf_{\delta})\}$ converge to fixed-points $\epsilon_*^x(\epsilon^x_{[0]},\dsf)$ and $\epsilon_*^x(\epsilon^x_{[0]},\dsf_{\delta})$ respectively.
	Therefore, \eqref{eq:k_distance} implies
	\begin{equation}
	\label{eq:all_d_sequences_converge}
	\begin{matrix}
	&\lim\limits_{\delta \to 0^{+}} \norm{\epsilon_*^x(\epsilon^x_{[0]},\dsf)-\epsilon_{\#}^x(\dsf_{\delta})}=0, & \forall \ \ \epsilon^x_{[0]} \in \mathcal{H}(\dsf_{\delta}).
	\end{matrix}
	\end{equation}
	where we used the fact that $\epsilon_{\#}^x(\dsf_{\delta})$ is unique and independent of $\epsilon^x_{[0]}$, as proven in Lemma~\ref{lem:trodden}.
	 Finally, we note that the function $\epsilon_*^x(\epsilon^x_{[0]},\dsf)$ is well defined for all $\dsf\geq \0$ from Lemma \ref{lem:brouwer}, and is continuous in $\dsf\geq \0$ since it is the composition of continuous functions. This implies that the limit $\lim\limits_{\delta \to 0^{+}} \epsilon_{\#}^x(\dsf_{\delta})$ is unique.
 Hence, \eqref{eq:all_d_sequences_converge} implies that there exists a unique fixed-point $\epsilon_{\#}^x(\dsf)$, i.e., 
	\begin{equation*}
	\begin{matrix}
	&&\epsilon_{\#}^x(\dsf):=\epsilon_*^x(\epsilon^x_{[0]},\dsf),&& \forall \ \ \epsilon^x_{[0]} \in \mathcal{H}(\dsf) \subset \mathcal{H}(\dsf_{\delta}),
	\end{matrix}
	\end{equation*}
	thus concluding the proof.
\end{proof}

Having shown that there exists a unique fixed-point $\epsilon_{\#}^x(\dsf)$, we proceed by replacing the lower-level optimization problems in \eqref{eq:bilevel_problem_to_solve_1_norm_innerApp} and \eqref{eq:bilevel_problem_to_solve_1_norm_outerApp} by the equivalent condition $\csf(\epsilon^x)+\dsf=\epsilon^x$.
\begin{remark}
	\label{remark:nonzero}
	While all the results presented assume that $\0 \in \WS$, there exist cases where it is not known a priori if the origin belongs to the disturbance set. Such cases can be accommodated in the formulation of Problems \eqref{eq:bilevel_problem_to_solve_1_norm_innerApp}-\eqref{eq:bilevel_problem_to_solve_1_norm_outerApp} by considering the disturbance set parametrization $\{\bar{w}\}\oplus\WS$, where $\0 \in \WS$ if $\epsilon^w \geq \0$, and $\bar{w}$ represents the origin offset. Then, an RPI set parametrized as $\{\bar{x}\}\oplus \FS$ satisfies $\{\A \bar{x}+\B \bar{w}\} \oplus \A\FS \oplus \B\WS \subseteq \{\bar{x}\}\oplus \FS$. From basic properties of support functions, this inclusion is equivalent to $E\A \bar{x}+E\B \bar{w}-E\bar{x}+\csf(\epsilon^x)+\dsf(\epsilon^w) \leq \bsf(\epsilon^x)$. The first part of this inequality can then be eliminated by using the steady-state offset value $\bar{x}=(\I-\A)^{-1}\B\bar{w}$. Hence, by appending the optimization variables $\bar{w}$ and $\bar{x}$, along with the equality constraint $\bar{x}=(\I-\A)^{-1}\B\bar{w}$, Problems \eqref{eq:bilevel_problem_to_solve_1_norm_innerApp},\eqref{eq:bilevel_problem_to_solve_1_norm_outerApp} can be modified to accommodate disturbance sets not including the origin. Note that the corresponding output-set is then $\{\C \bar{x}+\D\bar{w}\}\oplus \mathbb{Y}(\epsilon^x,\epsilon^w)$: the inclusions with respect to $\Y$ should be modeled by considering this offset. Since this modification is relatively straightforward, we skip further details because of space constraints.
	 $\hfill\square$
\end{remark}
\section{Characterizing Hausdorff distance and encoding Inclusion constraints}
\label{sec:inclusion_con}
 In this section, we use the inclusion encoding formulation presented in \cite{Sadraddini2019} under the following assumption on the output-set $\mathcal{Y}=\{y:Gy\leq g\}$ that is slightly stronger than Assumption \ref{ass:feasibility_assumptions}.
\begin{assumption}
	\label{ass:Y_full_dimensional}
	The set $\mathcal{Y}$ is full-dimensional, i.e., there exists some $\hat{y} \in \R^{n_y}$ and a scalar $\hat{\epsilon}>0$ such that $\{\hat{y}\} \oplus \hat{\epsilon}\mathcal{B}^{n_y}_{2} \subset \Y$. $\hfill\square$
\end{assumption}
For the \textit{inner-approximation} problem, this assumption implies that vector $g>\0$. For the \textit{outer-approximation problem}, this assumption implies system \eqref{eq:system} must be output-controllable for feasibility.
\subsection{Inner-approximation Problem \eqref{eq:bilevel_problem_to_solve_1_norm_innerApp}}
We use the Hausdorff distance given by
\begin{align}
\label{eq:dist_inner_defn}
\scalemath{0.94}{d_{\mathrm{H}}(\mathbb{Y}(\epsilon^x,\epsilon^w),\mathcal{Y})=\min_{\epsilon\geq \0}\left\{\sum_{j=1}^{m_B} \epsilon_j  \text{ s.t. } \mathcal{Y}\subseteq\mathbb{Y}(\epsilon^x,\epsilon^w) \oplus \mathbb{B}(\epsilon)\right\}},
\end{align}
for Problem \eqref{eq:bilevel_problem_to_solve_1_norm_innerApp}, where $\mathbb{B}(\epsilon):=\{y:H y \leq \epsilon\}$ with $\epsilon \in \R^{m_B}$, and the vectors $H_j^{\top}$ are sampled from the surface of $\mathcal{B}^{n_y}_{2}$. This choice of Hausdorff distance offers a practical way to model coverage of the set $\Y$ in directions indicated by the rows of $H$.  

In order to encode the inclusion $\mathcal{Y}\subseteq\mathbb{Y}(\epsilon^x,\epsilon^w) \oplus \mathbb{B}(\epsilon)$, we use the sufficient conditions presented in \cite[Theorem 1]{Sadraddini2019}, which states that the inclusion holds if there exist variables $z^{\mathrm{I}}:=\{\Sigma^{\mathrm{I}},\Theta^{\mathrm{I}},\Pi^{\mathrm{I}}\}$ with
$\Sigma^{\mathrm{I}} \in \R^{(n_x+n_w+n_y)\times n_y}$, $\Theta^{\mathrm{I}} \in \R^{n_x+n_w+n_y}$ and $\Pi^{\mathrm{I}} \in \R^{(m_X+m_W+m_B)\times m_Y}$ satisfying  $(\epsilon^x,\epsilon^w,\epsilon,z^{\mathrm{I}}) \in \Xi^{\mathrm{I}}$, where
\begin{equation*}\scriptsize
\Xi^{\mathrm{I}}:=\begin{Bmatrix}(\epsilon^x,\epsilon^w,\epsilon,z^{\mathrm{I}}):\begin{matrix}
[\C \ \D \ \I]\Sigma^{\mathrm{I}}=\I_{n_y}, \ [\C \ \D \ \I]\Theta^{\mathrm{I}}=\0_{n_y},
 \\
\Pi^{\mathrm{I}} \geq \0, \ \ \Pi^{\mathrm{I}} G = \begin{bmatrix} E && \\ & F & \\ & & H \end{bmatrix}\Sigma^{\mathrm{I}}, \vspace{2pt} \\
\Pi^{\mathrm{I}} g \leq \begin{bmatrix}
\epsilon^x \\ \epsilon^w \\ \epsilon
\end{bmatrix}+\begin{bmatrix} E && \\ & F & \\ & & H \end{bmatrix}\Theta^{\mathrm{I}}
\end{matrix} \end{Bmatrix},
\end{equation*}
is a set of linear equality and inequality constraints, and $\Pi^{\mathrm{I}}$ is a matrix of nonnegative elements. Since we enforce $\epsilon^w\geq \0$, and $\epsilon^x \geq \0$ by construction, the set $\mathbb{Y}(\epsilon^x,\epsilon^w)$ is always nonempty. Then, there always exists some $\epsilon \geq \0$ such that the inclusion $\Y \subseteq \mathbb{Y}(\epsilon^x,\epsilon^w) \oplus \mathbb{B}(\epsilon)$ holds. Under Assumption \ref{ass:Y_full_dimensional}, it then follows from \cite[Theorem 1]{Sadraddini2019} that there always exist variables $z^{\mathrm{I}}$ such that $\Xi^{\mathrm{I}}$ is nonempty. Since this inclusion encoding is only sufficient, the computed value of Hausdorff distance is an upper-bound to the actual value.

In order to encode the inclusion $\mathbb{Y}(\epsilon^x,\epsilon^w) \subseteq \Y$, we use the support functions defined as
\begin{align*}
&\lsfy_k\left (\epsilon^x\right ):=h_{\C\FS}\left (G_k^\top\right ), && \msfy_k\left (\epsilon^w\right ):=h_{\D\WS}\left (G_k^\top\right ),
\end{align*}
for each $k \in \mathbb{I}_1^{m_Y}$ to enforce the inequality $\lsfy(\epsilon^x)+\msfy(\epsilon^w) \leq g$. 

Hence, using the definition in  \eqref{eq:dist_inner_defn}, RPI set equivalence in \eqref{eq:main_result_RPI}, and the proposed inclusion encodings, we write Problem \eqref{eq:bilevel_problem_to_solve_1_norm_innerApp} as
\begin{align}
\label{eq:bilevel_problem_to_solve_1_norm_innerApp_withHD}
\scalemath{0.9}{\min_{\epsilon^w \geq \0,\epsilon^x,\epsilon\geq \0,z^{\mathrm{I}}}} &\hspace{5pt} \scalemath{0.9}{\sum_{j=1}^{m_B} \epsilon_j} \\
& \hspace{-15pt}\text{ s.t. } \csf(\epsilon^x)+\dsf(\epsilon^w)=\epsilon^x, \nonumber\\
& \hspace{3pt} \lsfy(\epsilon^x)+\msfy(\epsilon^w) \leq g, \nonumber \\
& \hspace{3pt} (\epsilon^x,\epsilon^w,\epsilon,z^{\mathrm{I}}) \in \Xi^{\mathrm{I}}. \nonumber
\end{align}
\subsection{Outer-approximation Problem \eqref{eq:bilevel_problem_to_solve_1_norm_outerApp}}
Similar to \eqref{eq:dist_inner_defn}, we use the Hausdorff distance given by
\begin{align}
\label{eq:dist_outer_defn}
\scalemath{0.94}{d_{\mathrm{H}}(\mathbb{Y}(\epsilon^x,\epsilon^w),\mathcal{Y})=\min_{\epsilon\geq \0}\left\{\sum_{j=1}^{m_B} \epsilon_j  \text{ s.t. } \mathbb{Y}(\epsilon^x,\epsilon^w)\subseteq \mathcal{Y}\oplus \mathbb{B}(\epsilon)\right\}}
\end{align}
\noindent for Problem \eqref{eq:bilevel_problem_to_solve_1_norm_outerApp}.
Then, we approximately encode the inclusion $\mathbb{Y}(\epsilon^x,\epsilon^w)\subseteq \mathcal{Y}\oplus \mathbb{B}(\epsilon)$ using the support functions defined as
\begin{align*}
&\lsfb_j\left (\epsilon^x\right ):=h_{\C\FS}\left (H^{\top}_j\right ), && \msfb_j\left (\epsilon^w\right ):=h_{\D\WS}\left (H^{\top}_j\right ), \\
&\gsfb_j:=h_{\Y}\left (H^{\top}_j\right ),
\end{align*}
for each $j \in \mathbb{I}_1^{m_B}$ through the inequality $\lsfb(\epsilon^x)+\msfb(\epsilon^w)-\epsilon \leq \gsfb$. The approximation results from the fact that this condition is only necessary for the inclusion to hold. Hence, the computed value of Hausdorff distance is a lower-bound to the actual value.

In order to encode the inclusion $\Y \subseteq \oplus_{t=0}^N \C\A^t \B \WS \oplus \D \WS$, we again use the sufficient conditions presented in \cite[Theorem 1]{Sadraddini2019}, which states that the inclusion holds if there exist variables $z^{\mathrm{O}}:=\{\Sigma^{\mathrm{O}},\Theta^{\mathrm{O}},\Pi^{\mathrm{O}}\}$ with dimensions $\Sigma^{\mathrm{O}} \in \R^{(N+2)n_w \times n_y}$, $\Theta^{\mathrm{O}} \in \R^{(N+2)n_w}$, $\Pi^{\mathrm{O}} \in \R^{(N+2)m_W \times m_Y}$ satisfying $(\epsilon^w,z^{\mathrm{O}})\in \Xi^{\mathrm{O}}$, where
\begin{equation*}\scriptsize
\Xi^{\mathrm{O}}:=\begin{Bmatrix}(\epsilon^w,z^{\mathrm{O}}):\begin{matrix}
[\C\B \ \C\A\B \ \cdots \C\A^N\B \ \D]\Sigma^{\mathrm{O}} = \I_{n_y}, \\
[\C\B \ \C\A\B \ \cdots \C\A^N\B \ \D]\Theta^{\mathrm{O}} = \0_{n_y}, \\
\Pi^{\mathrm{O}} \geq \0, \ \ \Pi^{\mathrm{O}} G = (\I_{N+2} \otimes F)\Sigma^{\mathrm{O}}, \\
\Pi^{\mathrm{O}} g \leq (\1_{N+2} \otimes \epsilon^w)+(\I_{N+2} \otimes F)\Theta^{\mathrm{O}}
\end{matrix} \end{Bmatrix}
\end{equation*}
is a set of linear equality and inequality constraints. Here, $\otimes$ denotes the Kronecker product. The set $\Xi^{\mathrm{O}}$ is nonempty under Assumption \ref{ass:Y_full_dimensional} under the same reasoning as that for set $\Xi^{\mathrm{I}}$.

Hence, using the definition in  \eqref{eq:dist_outer_defn}, RPI set equivalence in \eqref{eq:main_result_RPI}, and the proposed inclusion encodings, we write Problem \eqref{eq:bilevel_problem_to_solve_1_norm_outerApp} as
\begin{align}
\label{eq:bilevel_problem_to_solve_1_norm_outerApp_withHD}
\scalemath{0.9}{\min_{\epsilon^w \geq \0,\epsilon^x,\epsilon\geq \0,z^{\mathrm{O}}}} &\hspace{5pt} \scalemath{0.9}{\sum_{j=1}^{m_B} \epsilon_j} \\
& \hspace{-15pt}\text{ s.t. } \csf(\epsilon^x)+\dsf(\epsilon^w)=\epsilon^x, \nonumber\\
& \hspace{3pt} \lsfb(\epsilon^x)+\msfb(\epsilon^w)-\epsilon \leq g^{\mathrm{O}}, \nonumber \\
& \hspace{3pt} (\epsilon^w,z^{\mathrm{O}})\in \Xi^{\mathrm{O}}. \nonumber
\end{align}
\section{Numerical optimization}
\label{sec:SQP-GS}
In this section, we adopt a penalty function approach to solve Problems \eqref{eq:bilevel_problem_to_solve_1_norm_innerApp_withHD}-\eqref{eq:bilevel_problem_to_solve_1_norm_outerApp_withHD}. To that end, we first note that  $\epsilon^w$ might be unbounded above in both these problems in case of a nonminimal representation of $\WS$. We tackle this issue by introducing the support function 
\begin{align*}
\qsf_t\left (\epsilon^w\right ):=h_{\WS}\left (F^{\top}_t\right ) && \text{for each $t \in \mathbb{I}_1^{m_W}$},
\end{align*}
such that $\qsf(\epsilon^w)=\epsilon^w$ if and only if $\WS$ is in minimal representation.
Then, we modify the objective function of Problems \eqref{eq:bilevel_problem_to_solve_1_norm_innerApp_withHD}-\eqref{eq:bilevel_problem_to_solve_1_norm_outerApp_withHD} as
\begin{align}
\label{eq:minimal_W_obj}
\sum_{j=1}^{m_B} \epsilon_j + \sigma \sum_{t=1}^{m_W}(\epsilon^w_t - \qsf_t(\epsilon^w)),
\end{align}
where $\sigma>0$ is some scalar tuning parameter. This modification ensures that ($a$) the solution $\epsilon^w$ is such that $\WS$ is in a minimal representation; ($b$) the solution is not perturbed, since $\WS[\qsf(\epsilon^w)]=\WS$. 
This modification is not required for $\epsilon^x$, since the RPI constraint enforces uniqueness of $\epsilon^x$ for a given value of $\epsilon^w$. 

Then, we propose to use the penalty function approach presented in \cite{Anandalingam1990} to solve the problems with the modified objective function. In the rest of this section, we present the approach for the \textit{inner-approximation} problem. Since a very similar method follows for the \textit{outer-approximation} problems, we skip further details because of space constraints. 
Considering Problem \eqref{eq:bilevel_problem_to_solve_1_norm_innerApp_withHD} along with objective function \eqref{eq:minimal_W_obj}, we note that all the LPs formulating the support functions are feasible and bounded for every bounded $\epsilon^x \geq \0$ and $\epsilon^w \geq \0$. Hence, they satisfy strong duality \cite{Bard2006}. This property is exploited in the penalty function algorithm to compute local optima. Introducing the optimal primal and dual variables
\begin{align*}
\begin{matrix}
\text{LP} & \csf_i(\epsilon^x) & \dsf_i(\epsilon^w) & \lsfy_k(\epsilon^x) & \msfy_k(\epsilon^w) & \qsf_t(\epsilon^w) \\\hline
\text{Primal} & \za^{\csf_i} & \za^{\dsf_i} & \za^{\lsfy_k} & \za^{\msfy_k} & \za^{\qsf_t} \\
\text{Dual} & \lambda^{\csf_i} & \lambda^{\dsf_i} & \lambda^{\lsfy_k} & \lambda^{\msfy_k} & \lambda^{\qsf_t}
\end{matrix}
\end{align*} 
strong duality of the LPs implies that these values satisfy the primal and dual feasibility conditions, and have a zero duality gap.
For the support function $\csf_i(\epsilon^x)$, these conditions for a given $\epsilon^x \geq \0$ are $E \za^{\csf_i} \leq \epsilon^x$, $E^{\top} \lambda^{\csf_i} = A^{\top} E_i ^{\top}$, $\lambda^{\csf_i} \geq \0_{m_X}$ and ${\lambda^{\csf_i}}^{\top} \epsilon^x = E_i \A \za^{\csf_i}$. Introducing these variables along with the optimality conditions, Problem \eqref{eq:bilevel_problem_to_solve_1_norm_innerApp_withHD} is reformulated to a single-level problem. Within this reformulation, a penalty function approach is followed to penalize the duality gap using a constant $\mathbf{K}>0$ to obtain
\begin{align}
\label{eq:single_level_inner}
\scalemath{0.9}{\min_{\epsilon^w,\epsilon^x,\epsilon,z^{\mathrm{I}},\za^*,\lambda^*}} &\hspace{5pt} \scalemath{0.9}{\sum_{j=1}^{m_B} \epsilon_j + \sigma \sum_{t=1}^{m_W}(\epsilon^w_t - F_t \za^{\qsf_t})+\mathbf{K}\mathcal{P}(\epsilon^x,\epsilon^w,\za^*,\lambda^*)} \\
& \hspace{-15pt}\text{ s.t. } E_i \A \za^{\csf_i}+E_i \B \za^{\dsf_i}=\epsilon^x_i, &&\hspace{-40pt} \forall \ i \in \mathbb{I}_1^{m_X}, \nonumber\\
& \hspace{3pt} G_k \C \za^{\lsfy_k}+G_k \D \za^{\msfy_k} \leq g_k, &&\hspace{-40pt} \forall \ k \in \mathbb{I}_1^{m_Y}, \nonumber \\
& \hspace{3pt} (\epsilon^x,\epsilon^w,\epsilon,z^{\mathrm{I}}) \in \Xi^{\mathrm{I}},  \nonumber \\
& \hspace{3pt} E \za^{\csf_i}\leq \epsilon^x, \ \ E^{\top} \lambda^{\csf_i} = \A^{\top} E_i^{\top}, &&\hspace{-40pt} \forall \ i \in \mathbb{I}_1^{m_X}, \nonumber\\
& \hspace{3pt} F \za^{\dsf_i}\leq \epsilon^w, \ \ F^{\top} \lambda^{\dsf_i} = \B^{\top} E_i^{\top}, &&\hspace{-40pt} \forall \ i \in \mathbb{I}_1^{m_X}, \nonumber\\
& \hspace{3pt} E \za^{\lsfy_k}\leq \epsilon^x, \ \ E^{\top} \lambda^{\lsfy_k} = \C^{\top} G_k^{\top}, &&\hspace{-40pt} \forall \ k \in \mathbb{I}_1^{m_Y}, \nonumber\\
& \hspace{3pt} F \za^{\msfy_k}\leq \epsilon^w, \ \ F^{\top} \lambda^{\msfy_k} = \D^{\top} G_k^{\top}, &&\hspace{-40pt} \forall \ k \in \mathbb{I}_1^{m_Y}, \nonumber\\
& \hspace{3pt} F \za^{\qsf_t}\leq \epsilon^w, \ \ F^{\top} \lambda^{\qsf_t} = F_t^{\top}, &&\hspace{-40pt} \forall \ t \in \mathbb{I}_1^{m_W}, \nonumber \\
& \hspace{3pt} \epsilon^x \geq \0, \ \epsilon^w \geq \0, \ \epsilon \geq \0, \ \lambda^* \geq \0, \nonumber
\end{align}
where $z^*\in \R^{m_X(n_x+n_w)+m_Y(n_x+n_w)+m_Wn_w}$ denotes the optimal primal variables, $\lambda^*\in \R^{m_X(m_X+m_W)+m_Y(m_X+m_W)+m_W^2}$ denotes the optimal dual variables, and the penalty function
\begin{align*}
\scalemath{1}{
	\begin{matrix}\mathcal{P}(\epsilon^x,\epsilon^w,\za^*,\lambda^*):=\sum_{t=1}^{m_W}(\lambda^{\qsf_t^{\top}}\epsilon^w - F_t \za^{\qsf_t})+ \\
	&\hspace{-135pt}\sum_{i=1}^{m_X}(\lambda^{\csf_i^{\top}}\epsilon^x - E_i \A \za^{\csf_i}+\lambda^{\dsf_i^{\top}}\epsilon^w - E_i \B \za^{\dsf_i})+ \\
	&\hspace{-130pt}\sum_{k=1}^{m_Y}(\lambda^{\lsfy_k^{\top}}\epsilon^x - G_k \C \za^{\lsfy_k}+\lambda^{\msfy_k^{\top}}\epsilon^w - G_k \D \za^{\msfy_k})\end{matrix}
}
\end{align*}
penalizes the duality gap of the support function LPs. We denote Problem \eqref{eq:single_level_inner} as $\mathcal{F}(\epsilon^x,\epsilon^w,\epsilon,z^{\mathrm{I}},\za^*,\lambda^*,\mathbf{K})$. 
The main idea behind the approach presented in \cite{Anandalingam1990} is that there exists a penalty parameter $\mathbf{K}^*$ such that, if Problem \eqref{eq:single_level_inner} is solved with $\mathbf{K} > \mathbf{K}^*$, then the duality gap $\mathcal{P}(\epsilon^x,\epsilon^w,\za^*,\lambda^*)=0$ at the optimal solution, and this solution also solves the original Problem \eqref{eq:bilevel_problem_to_solve_1_norm_innerApp_withHD} with objective function \eqref{eq:minimal_W_obj}. Then in order to solve Problem \eqref{eq:single_level_inner}, an iterative algorithm is proposed, with each iteration composed of solving two LPs.
%

Denoting an iteration by the subscript $\{l\}$, the first LP solved is $\mathcal{F}(\epsilon^x,\epsilon^w,\epsilon,z^{\mathrm{I}},\za^*,\lambda^*_{\{l-1\}},\mathbf{K}_{\{l-1\}})$. Using the solution $(\epsilon_{\{l\}}^x,\epsilon_{\{l\}}^w,\epsilon_{\{l\}},z_{\{l\}}^{\mathrm{I}},\za_{\{l\}}^*)$ of this problem, the next step consists of solving the LP $\mathcal{F}(\epsilon_{\{l\}}^x,\epsilon_{\{l\}}^w,\epsilon_{\{l\}},z_{\{l\}}^{\mathrm{I}},\za_{\{l\}}^*,\lambda^*,\mathbf{K}_{\{l-1\}})$ for the variables $\lambda^*_{\{l\}}$. Finally, if the obtained values solve Problem \eqref{eq:single_level_inner} and duality gap is zero, the algorithm is terminated. Else, the procedure is repeated with $\mathbf{K}_{\{l\}} \geq \mathbf{K}_{\{l-1\}}$. This algorithm was shown to converge to a local optimal solution of Problem \eqref{eq:bilevel_problem_to_solve_1_norm_innerApp_withHD} in \cite{Anandalingam1990}. The approach was further extended in \cite{White1993} to obtain the global minimizer, and future work focuses on an efficient implementation of this method.
\begin{remark}
	We propose to initialize the optimization algorithm using the scaling $\zeta\geq0$ as $\epsilon^w_{\{0\}}=\zeta \1$ and $\epsilon^x_{\{0\}}=\zeta \hat{\epsilon}^x_{\{0\}}$, where $\hat{\epsilon}^x_{\{0\}}$ satisfies $\csf(\hat{\epsilon}^x_{\{0\}})+\dsf(\1)=\hat{\epsilon}^x_{\{0\}}$. This value can be computed using the one-step procedure in \cite{Trodden2016}, and $\zeta$ can be selected by solving an LP that enforces desired inclusions with respect to the output-set $\mathcal{Y}$. The dual variables corresponding to these LPs can be set as $\lambda^*_{\{0\}}$. We skip further details because of space constraints. $\hfill\square$
\end{remark}
	\begin{remark}
		Alternative procedures to compute the disturbance sets can be derived from the methods presented in \cite{Tahir2013,Tahir2015,Chengyuan2019}, by formulating their optimization problems with $\epsilon^w$ as an optimization variable and enforcing a fixed feedback gain. While the formulation in \cite{Tahir2013,Tahir2015} involves solving LPs, the \textit{reduced-complexity} polytopes can be excessively conservative. The formulation in \cite{Chengyuan2019} accommodates \textit{full-complexity} polytopes, with the nonlinear terms in the resulting optimization problem dealt with through a Newton-type procedure. Comparison with this method is a subject of future investigation.
 $\hfill\square$
	\end{remark}
\section{Numerical examples}
\label{sec:numerical_example}
We now present two examples, with the first related to the \textit{inner-approximation} problem, and the second to the  \textit{outer-approximation} problem. These results are obtained using the penalty function algorithm discussed in the previous section. CPLEX LP solver \cite{cplex2009v12} was used to solve the LPs. The sets are plotted using plotting tools from the Multi-Parametric Toolbox \cite{MPT3}. For both the examples, we use the set $\mathbb{B}(\epsilon)=\{y:\begin{bmatrix} \I & -\I \end{bmatrix}^{\top} y \leq \epsilon\}$ to define the Hausdorff distance in \eqref{eq:dist_inner_defn}-\eqref{eq:dist_outer_defn}, and choose $\sigma=1$ in \eqref{eq:minimal_W_obj}.
\subsection{Computation of safe reference-sets for supervisory control}
\begin{figure}[t]
	\centering
	\hspace{0cm}
	\resizebox{1\linewidth}{!}{
		\begin{tikzpicture}
		\begin{scope}[shift={(0,0)}]
		\node[draw=none,fill=none](ex_1) at (0,0) {\includegraphics[trim=218 130 216 125,clip,scale=0.7]{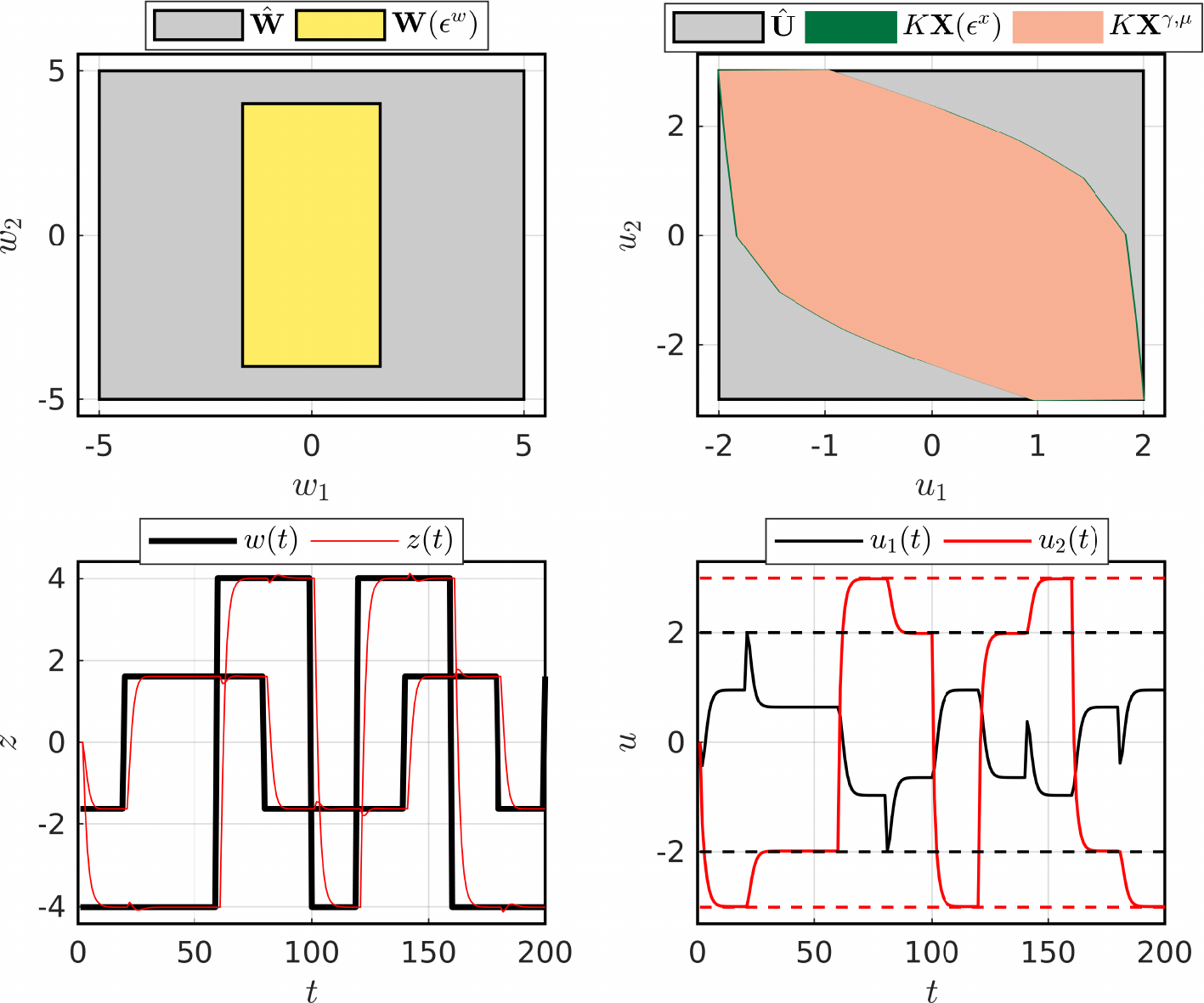}};
		\node[](box_W) [above right=-0.45cm and -6.92cm of ex_1,fill=white,minimum height=0.28cm,minimum width=0.73cm] {$$ };
		\node[](text_W) [above right=-0.56cm and -7.02cm of ex_1,fill=none,minimum size=0.2cm] {$\scalemath{0.7}{\WS}$ };
		\node[](box_X1) [above right=-0.32cm and 3.5cm of box_W,fill=white,minimum height=0.28cm,minimum width=0.9cm] {$$ };
		\node[](text_X1) [above right=-0.49cm and 3.4cm of text_W,fill=none,minimum size=0.2cm] {$\scalemath{0.7}{K\FS}$ };
		\node[](box_X2) [above right=-0.3cm and 0.81cm of box_X1,fill=white,minimum height=0.28cm,minimum width=0.8cm] {$$ };
		\node[](text_X2) [above right=-0.41cm and 0.75cm of text_X1,fill=none,minimum size=0.2cm] {$\scalemath{0.7}{K\mathcal{X}^{\gamma,\mu}}$ };
		\end{scope}
		\end{tikzpicture}}
	\captionsetup{width=1\linewidth}
	\caption{Results of solving the \textit{inner-approximation} problem. Tight RPI set is computed with parameters $\gamma=10^{-5}$, $\mu = 10^{-6}$. Bottom-left plot shows the tracking performance with $w$ sampled from the vertices of $\WS$. Bottom-right plot shows resulting closed-loop input response. Observe that the input bounds are respected. }
	\label{figure:R1_inner_problem_supervisor_integrator}
	\end{figure}
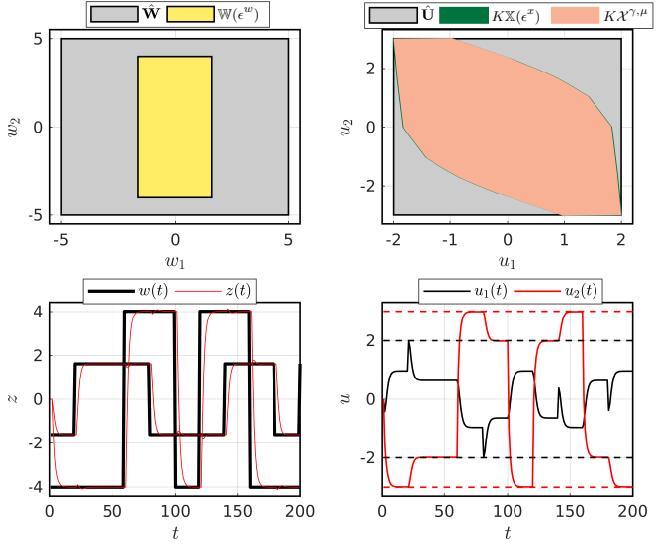
	We consider the system
	\begin{equation*}\small
	z(t+1)=\begin{bmatrix} 1.1 &0.2\\ -0.3& 0.4\end{bmatrix} z(t)+\begin{bmatrix} 1& 0 \\ 0.1 & 1 \end{bmatrix} u(t)
	\end{equation*}
	with input-constraints $u \in\hat{\mathbf{U}}:=\{u:|u| \leq \begin{bmatrix} 2 & 3 \end{bmatrix}^{\top}\}$. We assume that the system is equipped with an LQI-tracking controller such that $z$ tracks a reference signal $w$: an integral-action state $q$ with dynamics $q(t+1)=q(t)+z(t)-w(t)$ is appended, and the state $x = [z^{\top} \  q^{\top}]^{\top}$ is introduced. Then, an LQI feedback gain $K=\begin{bmatrix}-1.19&   -0.1439 &  -0.3154  &  0.0213\\
	0.2777  & -0.6497 &  -0.0037  & -0.3724\end{bmatrix}$ is computed corresponding to matrices $Q=\begin{bmatrix} \I_2 & \0 \\ \0 & 0.5\I_2 \end{bmatrix}$  and $R=\I_2$. The resulting closed-loop system with $u=Kx$ has the dynamics
	\begin{equation*}
	x(t+1)=\scalemath{0.75}{\begin{bmatrix}
		-0.09&  0.0561  & -0.3154  &  0.0213\\
		-0.1413  & -0.2641 &  -0.0353  & -0.3702\\
		1&     0   & 1 &        0\\
		0 &   1 &        0  &  1
		\end{bmatrix}}x(t)+\scalemath{0.75}{\begin{bmatrix} 0 & 0 \\ 0 & 0 \\ -1 & 0 \\ 0 & -1 \end{bmatrix}} w(t).
	\end{equation*}
	
	For this system, we aim to design a supervisory controller that saturates the reference signal such that input-constraints are respected: we compute the largest reference saturation limits $\epsilon^w =[\bar{w}_1 \ \ \bar{w}_2]^{\top}$ such that for all $w \in \WS=\{w:|w_1|\leq \bar{w}_1, |w_2| \leq \bar{w}_2\}$, we have $u \in \hat{\mathbf{U}}$. Moreover, the supervisory controller does not have access to the state $x(t)$ of the system, which implies the set $\WS$ should guarantee input-constraint satisfaction for all reachable $x$.

	In order to compute these bounds, we note that if $w \in \WS$, the state of the closed-loop system always belongs to the mRPI set as $x \in \mathcal{X}_{\mathrm{m}}(\epsilon^w)$ (provided $x(0) \in \mathcal{X}_{\mathrm{m}}(\epsilon^w)$). Then, the condition $u \in \hat{\mathbf{U}}$ in equivalent to the inclusion $K \mathcal{X}_{\mathrm{m}}(\epsilon^w) \subseteq \hat{\mathbf{U}}$. Finally, we assume that the references are always bounded as $w \in \hat{\mathbf{W}}:=\{w:|w| \leq \begin{bmatrix} 5 & 5 \end{bmatrix}^{\top}\}$. Hence, we compute the vector $\epsilon^w$ such that the inclusions $K \mathcal{X}_{\mathrm{m}}(\epsilon^w) \subseteq \hat{\mathbf{U}}$ and $\WS \subseteq \hat{\mathbf{W}}$ hold. We write
	\begin{equation*}
	y(t)=\begin{bmatrix} K \\ \0 \end{bmatrix} x(t) + \begin{bmatrix} \0 \\ \I \end{bmatrix} w(t), \text{ and the output-set $\mathcal{Y}=\hat{\mathbf{U}} \times \hat{\mathbf{W}},$}
	\end{equation*}
	based on which we solve the \textit{inner-approximation} problem \eqref{eq:bilevel_problem_to_solve_1_norm_innerApp_withHD}: We approximate the mRPI set using the RPI set $\FS=\{x:Ex \leq \epsilon^x\}$, where the matrix $E$ is composed of hyperplanes defining the set $\oplus_{t=0}^5 \A^t \B \hat{\mathbf{W}}$  ($\A$,$\B$ denote the matrices of the closed-loop system). This choice results in $m_X = 240$. The result of solving this problem using the methods presented in this paper is shown in Figure \ref{figure:R1_inner_problem_supervisor_integrator}. The computed saturation bounds are $\bar{w}_1=1.6172$, $\bar{w}_2=4.0125$.
	
	 We also plot the set $K \mathcal{X}^{\gamma,\mu}$, where $\mathcal{X}^{\gamma,\mu}$ is a tight RPI approximation of the mRPI set $\mathcal{X}_{\mathrm{m}}(\epsilon^w)$ od presented in \cite{Rakovic2006}: This method considers the system $x(t+1)=Ax(t)+\tilde{w}(t)$ with $\tilde{w} \in \B\WS \oplus \gamma \mathcal{B}^4_{\infty}$. Labeling the mRPI set for this system as $\mathcal{X}_{\mathrm{m}}^\gamma(\epsilon^w)$, the tightly approximating RPI set satisfies $\mathcal{X}^{\gamma,\mu} \subseteq \mathcal{X}_{\mathrm{m}}^\gamma(\epsilon^w) \oplus \mu \mathcal{B}^4_{\infty}$. We observe that $K\FS$ characterizes a fairly tight approximation of the set $K\mathcal{X}_{\mathrm{m}}(\epsilon^w)$. Closed-loop trajectories are plotted with references $w$ sampled from the vertices of $\WS$, for which the input response satisfies the input-constraints. Hence, if $x(0) \in \FS$, the supervisory controller can command any reference $w \in \WS$ with guaranteed input-constraint satisfaction.
\begin{remark}
The mRPI set is suitable to formulate the problem in Example A since we do not have access to the state $x(t)$. If this limitation is overcome, then a reference governor scheme \cite{Garone2017} is more suitable to design the supervisory controller, which uses control invariant sets to guarantee constraint satisfaction. $\hfill\square$
\end{remark}
\subsection{Computation of input-constraint sets for output reachability}
We consider system \eqref{eq:system} with initial-state $x(0)=\0$, for which we compute the smallest input-constraint set $\WS=\{w:Fw\leq \epsilon^w\}$ with rows $F_i=[\mathrm{sin}(2\pi(i-1)/m_W) \ \mathrm{cos}(2\pi(i-1)/m_W) ]$
for each $i \in \mathbb{I}_1^{m_W}$, such that all $y \in \Y$ can be reached with control inputs $w \in \WS$. To that end, we use $\fsf(\epsilon^w)$ as a measure of the set $\WS$, and formulate the optimization problem $\mathbb{P}^N$ defined as
\begin{align*}
\epsilon^{w,N} := &\arg\min_{\epsilon^w \geq \0} 
\ \ \fsf(\epsilon^w)\\
&\qquad \hspace{5pt} \text{s.t.} \hspace{10pt} \scalemath{0.95}{y = \Sigma_{t=0}^{N-1} \C\A^t \B w_y(t)+\D w_y(N),} \ \ && \\
& \qquad \hspace{25pt} \scalemath{0.95}{w_y(t) \in \WS,} \ \scalemath{0.95}{\forall \ y \in \mathcal{Y},} \ \ \scalemath{0.95}{\forall \ t \in \mathbb{I}_0^{N},}
\end{align*}
such that $\WS[\epsilon^{w,N}]$ is the smallest input-constraint set in which there exist inputs driving the output of system \eqref{eq:system} to all $y \in \Y$ from the origin in $N$-steps. If Assumption \ref{ass:feasibility_assumptions}-\textit{Outer} holds, then $\mathbb{P}^N$ is feasible for all $N\geq n_x$. It can then be shown that the sequence of optimal values $\{\fsf(\epsilon^{w,N})\}_N$ is non-increasing, and converges to the optimal value of the problem
\begin{align}
\label{eq:smallest_input_set_prob}
\epsilon_*^{w} := \arg\min_{\epsilon^w \geq \0} 
\ \ \fsf(\epsilon^w)
\quad \hspace{5pt} \text{s.t.} \hspace{10pt} \scalemath{0.9}{ \Y \subseteq \C\mathcal{X}_{\mathrm{m}}(\epsilon^w)\oplus \D\WS},
\end{align}
where $\mathcal{X}_{\mathrm{m}}(\epsilon^w)$ is the mRPI set corresponding to $\WS$. This follows from the idea that the mRPI set is the closure of the largest $0$-reachable set \cite{Blanchini2015}.
Hence, computing the smallest input-constraint set entails solving Problem \eqref{eq:smallest_input_set_prob}. We choose $\fsf(\epsilon^w)=d_{\mathrm{H}}(\mathcal{Y}_{\mathrm{m}}(\epsilon^w),\Y)$, such that Problem \eqref{eq:smallest_input_set_prob} is equivalent to Problem \eqref{eq:orig_problem_to_solve} along with the output-set inclusion constraint. This choice ensures that we compute an input-constraint set $\WS$ whose $0$-reachable set in the output space tightly includes the target output-set $\Y$.
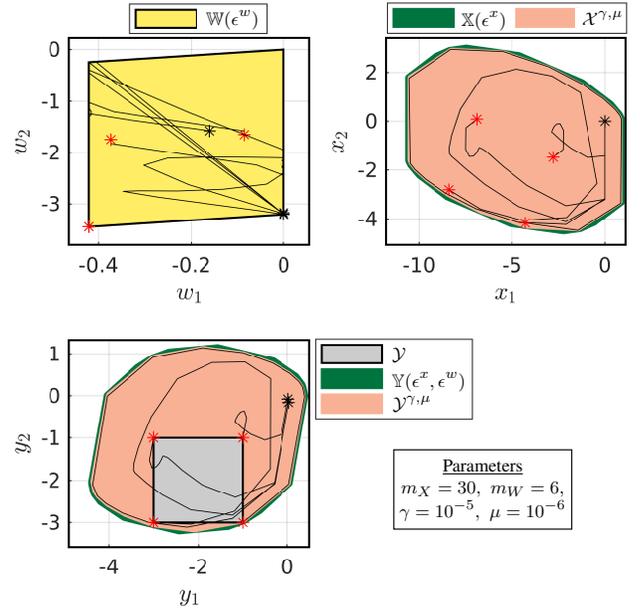
\begin{figure}[t]
	\centering
	\hspace{0cm}
	\vspace{-0.cm}
	\resizebox{1\linewidth}{!}{
		\begin{tikzpicture}
		\begin{scope}[shift={(0,0)}]
		\node[draw=none,fill=none](ex_1) at (0,0) {\includegraphics[trim=260 150 265 146,clip,scale=0.7]{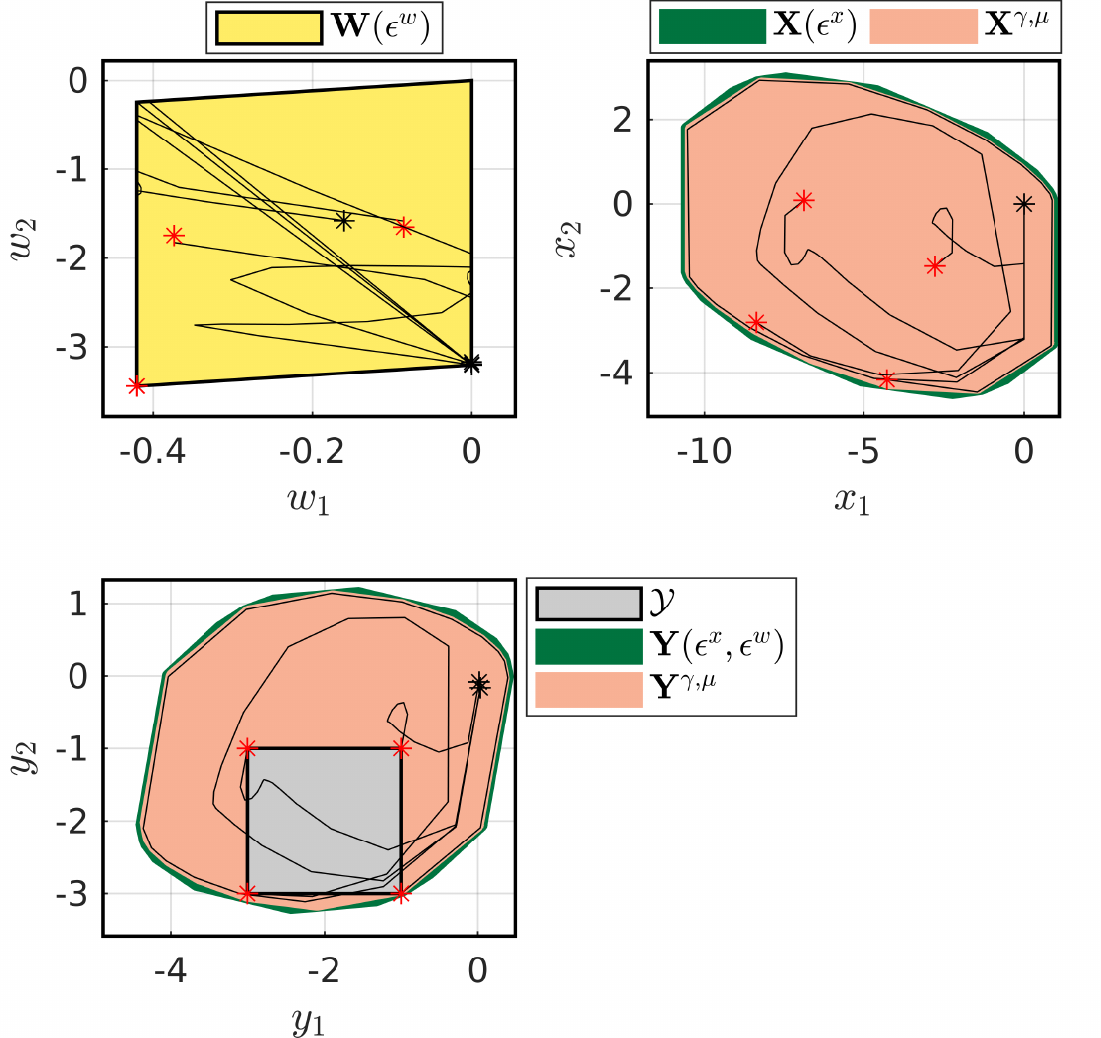}};
		\node[](box_W) [above right=-0.455cm and -5.65cm of ex_1,fill=white,minimum height=0.28cm,minimum width=0.8cm] {$$ };
		\node[](text_W) [above right=-0.55cm and -5.65cm of ex_1,fill=none,minimum size=0.2cm] {$\scalemath{0.7}{\WS}$ };
		\node[](box_X1) [above right=-0.3cm and 2.35cm of box_W,fill=white,minimum height=0.28cm,minimum width=0.7cm] {$$ };
		\node[](text_X1) [above right=-0.47cm and 2.25cm of text_W,fill=none,minimum size=0.2cm] {$\scalemath{0.7}{\FS}$ };
		\node[](box_X2) [above right=-0.3cm and 0.75cm of box_X1,fill=white,minimum height=0.28cm,minimum width=0.55cm] {$$ };
		\node[](text_X2) [above right=-0.4cm and 0.7cm of text_X1,fill=none,minimum size=0.2cm] {$\scalemath{0.7}{\mathcal{X}^{\gamma,\mu}}$ };
		\node[](box_Y) [below right=3.82cm and 1.5cm of box_W,fill=white,minimum height=0.9cm,minimum width=1cm] {$$ };
		\node[](text_Y) [above right=-0.34cm and -1.05cm of box_Y,fill=none,minimum size=0.2cm] {$\scalemath{0.7}{\mathcal{Y}}$ };
		\node[](text_Yo) [below right=-0.1cm and -0.4cm of text_Y,fill=none,minimum size=0.2cm] {$\scalemath{0.7}{\mathbb{Y}(\epsilon^x,\epsilon^w)}$ };
		\node[](text_Ym) [below right=-0.17cm and -1.12cm of text_Yo,fill=none,minimum size=0.2cm] {$\scalemath{0.7}{\mathcal{Y}^{\gamma,\mu}}$ };
		\node[](Details) [below right=0.6cm and -1cm of text_Yo,fill=white,draw=black,inner sep=2pt] { $\scalemath{0.7}{\begin{matrix} \text{\underline{Parameters}} \\m_X=30, \ m_W=6, \\ \gamma=10^{-5}, \ \mu = 10^{-6}\end{matrix}}$ };
		\end{scope}
		\end{tikzpicture}}
	\captionsetup{width=1\linewidth}
	\caption{Results of solving the  \textit{outer-approximation} problem. Input, state and output trajectories are plotted with $w(0),x(0),y(0)$ denoted by black $*$, $w(100),x(100),y(100)$ denoted by red $*$. Observe that the vertices of $\mathcal{Y}$ are reachable from $x(0)=\0$ with $w \in \WS$. }
	\label{figure:R1_outer_problem_simple}
\end{figure}
 We approximately solve Problem \eqref{eq:smallest_input_set_prob} based on the \textit{outer-approximation} formulation in Problem \eqref{eq:bilevel_problem_to_solve_1_norm_outerApp}: we approximate the mRPI set using the polytopic RPI set $\FS=\{x:Ex \leq \epsilon^x\}$  with rows $E_i=[\mathrm{sin}(2\pi(i-1)/m_X) \ \mathrm{cos}(2\pi(i-1)/m_X) ]$
 for each $i \in \mathbb{I}_1^{m_X}$. Using this set, we formulate Problem \eqref{eq:bilevel_problem_to_solve_1_norm_outerApp_withHD}. The results of solving this problem using the methods presented in this paper are shown in Figure \ref{figure:R1_outer_problem_simple}. We consider system \eqref{eq:system} with matrices
 \begin{align*}
 &\scalemath{0.98}{\A=\begin{bmatrix} 0.8966 & 0.8822 \\ -0.2068 & 0.3244 \end{bmatrix},\ \ \
 \B=\begin{bmatrix} 0 & 0 \\ -1 & 1 \end{bmatrix},} \\
 &\scalemath{0.98}{\C=\begin{bmatrix}
 0.4 & 0.1 \\ 0.1 & 0.6 \end{bmatrix}, \ \ \ \qquad \qquad
 \D=\begin{bmatrix}  0.001 &  -0.01 \\
 0.003&   0.05
 \end{bmatrix}},
 \end{align*} 
and the target output-set $\mathcal{Y}=\{[-2 \ -2]^{\top}\}\oplus \mathcal{B}^2_{\infty}$.
This system is the closed-loop form of the standard double-integrator with feedback gain $K=[0.2068    \ 0.6756]$.
We choose $N=100$ in the formulation of Problem \eqref{eq:bilevel_problem_to_solve_1_norm_outerApp_withHD}. We see that the computed set $\WS$ is such all $y \in \Y$ are reachable from the origin. We also plot tight approximation RPI set $\mathcal{X}^{\gamma,\mu}$ of the mRPI set $\mathcal{X}_{\mathrm{m}}(\epsilon^w)$ using the methods presented in \cite{Rakovic2006}, in a manner similar to the previous example. We observe through the set $\mathcal{Y}^{\gamma,\mu}:=\C\mathcal{X}^{\gamma,\mu} \oplus \D\WS$ that $\Y \subseteq \mathcal{Y}^{\gamma,\mu}\subseteq \mathbb{Y}(\epsilon^x,\epsilon^w)$ holds, thus ensuring the desired reachability.

In conclusion, one can design feedback controllers to select inputs $w$ from the input-constraint set $\WS$, with the guarantee that for any $x(0) \in \mathcal{X}_{\mathrm{m}}(\epsilon^w)$, there always exist feasible inputs to reach every target output $y \in \mathcal{Y}_{\mathrm{m}}(\epsilon^w)\supset \Y$. In Figure \ref{figure:R1_outer_problem_simple}, we also plot state, input and output trajectories with $x(0)=\0$ and $y(100) \in \Y$ to demonstrate the reachability.
	\section{Conclusions}
We have presented a method for computing an input disturbance set for discrete-time linear time-invariant systems such that the reachable set of outputs approximates an assigned set. To that end, we formulated an optimization problem in order to minimize the approximation error. Finally, we presented some numerical results to demonstrate the feasibility of the approach and two possible practical applications.
Future research will further develop the solution algorithm by considering: (a) alternative solution methods such as, e.g.,  value function approaches \cite{Bard2006}; (b) optimizing also over matrices $E$ and $F$. Finally, the potential of this technique when applied to feedback controller synthesis and to system identification problems will be investigated.

	\bibliography{references}

\newcommand{\noop}[1]{}
\begin{thebibliography}{10}

\bibitem{Blanchini2015}
F.~Blanchini and S.~Miani, {\em Set-Theoretic Methods in Control}.
\newblock 01 2007.

\bibitem{Blanchini1999}
F.~Blanchini, ``Set invariance in control,'' {\em Automatica}, vol.~35, no.~11,
  pp.~1747 -- 1767, 1999.

\bibitem{Bertsekas1971}
D.~Bertsekas and I.~Rhodes, ``On the minimax reachability of target sets and
  target tubes,'' {\em Automatica}, vol.~7, no.~2, pp.~233 -- 247, 1971.

\bibitem{Bertsekas1972}
D.~{Bertsekas}, ``Infinite time reachability of state-space regions by using
  feedback control,'' {\em IEEE Transactions on Automatic Control}, vol.~17,
  pp.~604--613, October 1972.

\bibitem{Kolmanovsky1998}
I.~Kolmanovsky and E.~G. Gilbert, ``Theory and computation of disturbance
  invariant sets for discrete-time linear systems,'' {\em Mathematical Problems
  in Engineering}, vol.~4, no.~4, pp.~317--367, 1998.

\bibitem{Rawlings2009}
J.~Rawlings and D.~Mayne, {\em Model Predictive Control: Theory and Design}.
\newblock 01 2009.

\bibitem{Kouvaritakis2015}
B.~Kouvaritakis and M.~Cannon, ``Model predictive control: Classical, robust
  and stochastic,'' 2015.

\bibitem{Garone2017}
E.~Garone, S.~D. Cairano, and I.~Kolmanovsky, ``Reference and command governors
  for systems with constraints: A survey on theory and applications,'' {\em
  Automatica}, vol.~75, pp.~306 -- 328, 2017.

\bibitem{Mayne2005}
D.~Mayne, M.~Seron, and S.~Rakovi\'c, ``Robust model predictive control of
  constrained linear systems with bounded disturbances,'' {\em Automatica},
  vol.~41, no.~2, pp.~219 -- 224, 2005.

\bibitem{RakovicThesis}
S.~V. Rakovi\'c, {\em Robust control of constrained discrete time systems:
  Characterization and implementation}.
\newblock PhD thesis, 01 2005.

\bibitem{Rakovic2005RCI}
S.~Rakovi\'c, D.~Mayne, E.~Kerrigan, and K.~Kouramas, ``Optimized robust
  control invariant sets for constrained linear discrete-time systems,'' {\em
  IFAC Proceedings Volumes}, vol.~38, no.~1, pp.~584 -- 589, 2005.
\newblock 16th IFAC World Congress.

\bibitem{Riverso2013}
S.~{Riverso}, M.~{Farina}, and G.~{Ferrari-Trecate}, ``Plug-and-play
  decentralized model predictive control for linear systems,'' {\em IEEE
  Transactions on Automatic Control}, vol.~58, pp.~2608--2614, Oct 2013.

\bibitem{Olaru2008}
S.~Olaru, J.~D. Don\'{a}, and M.~Seron, ``Positive invariant sets for fault
  tolerant multisensor control schemes,'' {\em IFAC Proceedings Volumes},
  vol.~41, no.~2, pp.~1224 -- 1229, 2008.
\newblock 17th IFAC World Congress.

\bibitem{Mulagaleti2021}
S.~K. {Mulagaleti}, A.~{Bemporad}, and M.~{Zanon}, ``Computation of
  least-conservative state-constraint sets for decentralized mpc with dynamic
  and constraint coupling,'' {\em IEEE Control Systems Letters}, vol.~5, no.~1,
  pp.~235--240, 2021.

\bibitem{Flores2008}
J.~Flores, D.~Eckhard, and J.~G. {da Silva}, ``On the tracking problem for
  linear systems subject to control saturation,'' {\em IFAC Proceedings
  Volumes}, vol.~41, no.~2, pp.~14168--14173, 2008.
\newblock 17th IFAC World Congress.

\bibitem{Odelson2006}
B.~Odelson, M.~Rajamani, and J.~Rawlings, ``A new autocovariance least-squares
  method for estimating noise covariances,'' {\em Automatica}, vol.~42,
  pp.~303--308, 02 2006.

\bibitem{Mulagaleti2020a}
S.~Mulagaleti, M.~Zanon, and A.~Bemporad, ``Dynamic output disturbance models
  for robust model predictive control,'' {\em Proceedings of the 21st IFAC
  World Congress}, 2020.
\newblock to be published.

\bibitem{Rakovic2013}
S.~V. Rakovi\'{c}, B.~Kouvaritakis, and M.~Cannon, ``Equi-normalization and
  exact scaling dynamics in homothetic tube model predictive control,'' {\em
  Systems \& Control Letters}, vol.~62, no.~2, pp.~209 -- 217, 2013.

\bibitem{Trodden2016}
P.~{Trodden}, ``A one-step approach to computing a polytopic robust positively
  invariant set,'' {\em IEEE Transactions on Automatic Control}, vol.~61,
  pp.~4100--4105, Dec 2016.

\bibitem{Anandalingam1990}
G.~{Anandalingam} and D.~J. {White}, ``A solution method for the linear static
  stackelberg problem using penalty functions,'' {\em IEEE Transactions on
  Automatic Control}, vol.~35, no.~10, pp.~1170--1173, 1990.

\bibitem{Rakovic2005}
S.~V. {Rakovi\'{c}}, E.~C. {Kerrigan}, K.~I. {Kouramas}, and D.~Q. {Mayne},
  ``Invariant approximations of the minimal robust positively invariant set,''
  {\em IEEE Transactions on Automatic Control}, vol.~50, pp.~406--410, March
  2005.

\bibitem{Lorenzetti2019}
J.~{Lorenzetti} and M.~{Pavone}, ``{A Simple and Efficient Tube-based Robust
  Output Feedback Model Predictive Control Scheme},'' {\em arXiv e-prints},
  p.~arXiv:1911.07360, Nov 2019.

\bibitem{Althoff2016}
M.~Althoff and G.~Frehse, ``Combining zonotopes and support functions for
  efficient reachability analysis of linear systems,'' 12 2016.

\bibitem{Batkovic2018}
I.~{Batkovic}, M.~{Zanon}, N.~{Lubbe}, and P.~{Falcone}, ``A computationally
  efficient model for pedestrian motion prediction,'' in {\em 2018 European
  Control Conference (ECC)}, pp.~374--379, June 2018.

\bibitem{Sadraddini2019}
S.~{Sadraddini} and R.~{Tedrake}, ``{Linear Encodings for Polytope Containment
  Problems},'' {\em arXiv e-prints}, p.~arXiv:1903.05214, Mar 2019.

\bibitem{Bard2006}
J.~F. Bard, {\em Practical Bilevel Optimization: Algorithms and Applications
  (Nonconvex Optimization and Its Applications)}.
\newblock Berlin, Heidelberg: Springer-Verlag, 2006.

\bibitem{White1993}
D.~White and A.~Anandalingam, ``A penalty function approach for solving
  bi-level linear programs,'' {\em Journal of Global Optimization}, vol.~3,
  pp.~397--419, 12 1993.

\bibitem{Tahir2013}
F.~Tahir and I.~M. Jaimoukha, ``Robust feedback model predictive control of
  constrained uncertain systems,'' {\em Journal of Process Control}, vol.~23,
  no.~2, pp.~189--200, 2013.
\newblock IFAC World Congress Special Issue.

\bibitem{Tahir2015}
F.~{Tahir} and I.~M. {Jaimoukha}, ``Low-complexity polytopic invariant sets for
  linear systems subject to norm-bounded uncertainty,'' {\em IEEE Transactions
  on Automatic Control}, vol.~60, no.~5, pp.~1416--1421, 2015.

\bibitem{Chengyuan2019}
C.~Liu, F.~Tahir, and I.~M. Jaimoukha, ``Full-complexity polytopic robust
  control invariant sets for uncertain linear discrete-time systems,'' {\em
  International Journal of Robust and Nonlinear Control}, vol.~29, no.~11,
  pp.~3587--3605, 2019.

\bibitem{cplex2009v12}
I.~I. Cplex, ``V12. 1: User’s manual for cplex,'' {\em International Business
  Machines Corporation}, vol.~46, no.~53, p.~157, 2009.

\bibitem{MPT3}
M.~Herceg, M.~Kvasnica, C.~Jones, and M.~Morari, ``{Multi-Parametric Toolbox
  3.0},'' in {\em Proc.~of the European Control Conference}, (Z\"urich,
  Switzerland), pp.~502--510, July 17--19 2013.
\newblock \url{http://control.ee.ethz.ch/~mpt}.

\bibitem{Rakovic2006}
S.~V. {Rakovi\'c} and K.~I. {Kouramas}, ``The minimal robust positively
  invariant set for linear discrete time systems: Approximation methods and
  control applications,'' in {\em Proceedings of the 45th IEEE Conference on
  Decision and Control}, pp.~4562--4567, Dec 2006.

\end{thebibliography}

{\color{magenta}
}
\end{document}